\documentclass[11pt,a4paper,reqno]{amsart}
\usepackage{amsfonts}
\usepackage{amsmath,amssymb,amsthm,amsxtra}
\usepackage{mathrsfs}
\usepackage{float}
\usepackage[colorlinks,linkcolor=blue!72,anchorcolor=orange,citecolor=red,urlcolor=Emerald, bookmarksopen,bookmarksdepth=2]{hyperref}
\usepackage[usenames,dvipsnames]{xcolor}
\usepackage{enumitem}
\usepackage{geometry,array}
\usepackage{graphicx}
\usepackage{subfigure}
\usepackage{bookmark}
\usepackage{tikz}
\usepackage{url}

\usetikzlibrary{matrix,positioning,decorations.markings,arrows,decorations.pathmorphing,	
    backgrounds,fit,positioning,shapes.symbols,chains,shadings,fadings,calc}
\tikzset{->-/.style={decoration={  markings,  mark=at position #1 with
    {\arrow{>}}},postaction={decorate}}}
\tikzset{-<-/.style={decoration={  markings,  mark=at position #1 with
    {\arrow{<}}},postaction={decorate}}}
\usepackage[all]{xy}
\usepackage[figuresright]{rotating}
\usepackage{makecell}
\usepackage{cleveref}
\usepackage{pdflscape}
\usetikzlibrary{patterns}
\usetikzlibrary{cd}
\usepackage{longtable}
\theoremstyle{plain}
\newtheorem{theorem}{Theorem}[section]
\newtheorem{lemma}[theorem]{Lemma}

\newtheorem{proposition}[theorem]{Proposition}
\theoremstyle{definition}
\newtheorem{definition}[theorem]{Definition}

\newtheorem{example}[theorem]{Example}
\newtheorem{remark}[theorem]{Remark}

\newtheorem{convention}[theorem]{Convention}
\numberwithin{equation}{section}

\numberwithin{equation}{section}

\newcommand\Old{\bgroup\markoverwith{\textcolor{red}{\rule[0.5ex]{2pt}{0.4pt}}}\ULon}

\def\<{\langle}
\def\>{\rangle}

\def\ZZ{\mathbb{Z}}

\newcommand{\oInt}{\overrightarrow{\operatorname{Int}}}
\newcommand{\oIntd}{\oInt^{\rho}}
\newcommand{\Hom}{\operatorname{Hom}\nolimits}

\newcommand{\Ext}{\operatorname{Ext}\nolimits}

\newcommand{\D}{\operatorname{\mathcal{D}}}

\def\numbers{\begin{enumerate}[label=\arabic*{$^\circ$}.]}
\def\ends{\end{enumerate}}

\def\wc{\widetilde{c}}
\def\wg{\widetilde{\gamma}}
\def\we{\widetilde{\eta}}
\def\ws{\widetilde{\sigma}}
\def\wt{\widetilde{\tau}}
\def\wa{\widetilde{\alpha}}
\def\wb{\widetilde{\beta}}

\def\h{\mathcal H}

\def\arrow{red}

\newcommand{\MCG}{\operatorname{MCG}}

\newcommand\Dehn[1]{\mathrm{D}_{#1}}

\def\dexc{red}

\def\M{\mathbf{M}}
\def\Y{\mathbf{Y}}

\def\ac{\mathbb{A}}

\renewcommand{\k}{\mathbf{k}}

\def\surf{\mathbf{S}}
\def\surfi{\surf^\circ}
\newcommand\Sp{\operatorname{Sp}}

\usetikzlibrary{decorations.markings}
\tikzset{->-/.style={decoration={  markings,  mark=at position #1 with
    {\arrow{>}}},postaction={decorate}}}
\tikzset{-<-/.style={decoration={  markings,  mark=at position #1 with
    {\arrow{<}}},postaction={decorate}}}

\def\nn{node{$\bullet$}}

\def\DEC{node[white]{$\bullet$}node[red]{$\circ$}}
\def\grad{\lambda}

\def\wCA{\widetilde{\operatorname{CA}}}

\def\gmsx{\gms\x}        
\def\x{_\vot}

\def\sg{\Lambda_{\mathbb{A}}}

\newcommand{\gms}{\surf^\grad}

\def\Dfang{\mathrm{D}^2_\vot}
\usepackage{marvosym}
\def\vot{\text{\Biohazard}}
\def\Vot{\text{\Cancer}}

\def\good{\mathcal{D}^@(\sg)}

\def\?{B}

\def\EG{\operatorname{EG}}

\def\xian{arc\;}

\def\CAS{\mathbb{S}}
\def\CA{\we}

\def\hua{\mathcal}

\newcommand{\tilt}[3]{{#1}^{#2}_{#3}}

\def\t{\tilde}

\def\beq{\begin{equation}}
\def\eeq{\end{equation}}

\def\ang{a}

\def\XX{T}

\tikzcdset{arrow style=tikz, diagrams={>=stealth}}
\begin{document}
\title[Quad as Stab of Graded Skew-gentle Algebras]{Quadratic Differentials as Stability Conditions of Graded Skew-gentle Algebras}
\author{Suiqi Lu}
\address{Lsq:
	Department of Mathematical Sciences,
	Tsinghua University,
    Beijing,
    China}
\email{lu-sq22@mails.tsinghua.edu.cn}

\author{Yu Qiu}
\address{Qy:
	Yau Mathematical Sciences Center and Department of Mathematical Sciences,
	Tsinghua University,
    Beijing,
    China.
    \&
    Beijing Institute of Mathematical Sciences and Applications, Yanqi Lake, Beijing, China}
\email{yu.qiu@bath.edu}

\author{Dongjian Wu}
\address{Wdj:
	Department of Mathematics, Graduate School of Science, The University of Osaka, Toyonaka Osaka, Japan}
\email{wu.dongjian.7cx@osaka-u.ac.jp}


\begin{abstract}
We prove that the principal component of the exchange graph of hearts of a graded skew-gentle algebra
can be identified with the corresponding exchange graph of S-graphs,
using the geometric models and the intersection formula in \cite{QZZ}.
Using the similar argument in \cite{BS, BMQS, CHQ},
we extend this identification to an isomorphism between the spaces of stability conditions and of quadratic differentials.
\end{abstract}


\maketitle
\tableofcontents
\addtocontents{toc}{\setcounter{tocdepth}{1}}

\setlength\parindent{0pt}
\setlength{\parskip}{5pt}

\section{Introduction}
The notion of a stability condition on a triangulated category was introduced in \cite{B1}, based on the study of slope stability of vector bundles over curves and $\Pi$-stability of D-branes in string theory. After the developments of many mathematicians, the theory of stability conditions plays an important role in many branches of mathematics, such as mirror symmetry, Donaldson-Thomas invariants and cluster theory, etc.

Independently, Kontsevich and Seidel suggested that moduli spaces of quadratic differentials might have a close relationship with spaces of stability conditions on Fukaya-type categories of surfaces.
Since then, many results have been obtained in this area, including \cite{BS,HKK,KQ2,IQ2,BMQS,CHQ,CHQ2}.

In this paper, our main focus is on establishing connections between quadratic differentials and stability conditions of graded skew-gentle algebras.
This is a generalization of the result (on unpunctured case) of Haiden-Katzarkov-Kontsevich (=HKK) \cite{HKK} to the punctured case.

\subsection{Topological Fukaya categories and graded skew-gentle algebras}
Fukaya categories are initially from symplectic geometry. They also play a crucial role in the study of homological mirror symmetry. In the classical homological mirror symmetry by Kontsevich \cite{K}, there exists a derived equivalence between a mirror pair $(X, X^{\vee})$ given as follows:
\begin{equation}\label{eq:HMS}
    \D\mathcal F(X)\cong \D^b\,{\rm Coh}\,(X^{\vee}).
\end{equation}
Here, the left-hand side (A-side) represents the derived Fukaya category of $X$, while the right-hand side (B-side) is the bounded derived category of coherent sheaves on the mirror dual $X^{\vee}$.

In this paper, we will study a type of Fukaya categories: the topological Fukaya categories of Riemann surfaces with punctures.
In the unpunctured case, HKK \cite{HKK} constructed these categories as the topological Fukaya categories of graded marked surfaces $\surf^{\lambda}$.
They proved that the endomorphism algebra of any formal generators of the topological Fukaya category $\mathcal F(\mathbf{S}^{\lambda})$ is a graded gentle algebra.
Moreover, they gave a complete classification of objects and showed that the isomorphism classes of indecomposable objects correspond one-to-one to the equivalence classes of admissible curves equipped with local systems on $\mathbf{S}^{\lambda}$.
Hence $\surf^{\lambda}$ provides a geometric model for graded gentle algebras.
As applications, there are many study on the derived invariants of graded gentle algebras using geometric models, e.g. \cite{LP1,APS19}.
In \cite{IQZ}, a partial intersection formulas, i.e.
dimensions of homomorphism spaces between objects equal the intersection between the corresponding arcs is given (cf. also \cite{OPS18}).

In the punctured case, the topological Fukaya categories are given by derived categories of skew-gentle algebras, which were first introduced by Gei{\ss}-de la Pe\~{n}a \cite{GP}, as a generalization of gentle algebras.
These algebras are an important class of representation-tame finite dimensional algebras.
As typical examples of gentle algebras are type $A$ and $\widetilde{A}$ algebras,
the examples for skew-gentle algebras are type $D$ and $\widetilde{D}$ algebras.
The classification of indecomposable modules for skew-gentle algebras, as well as clannish algebras in general, has been achieved by Crawley-Boevey \cite{CB89}, Deng \cite{Deng} and Bondarenko \cite{B},
cf. also Bekkert-Marcos-Merklen \cite{BMM} and Burban-Drozd \cite{BD}.

Similarly to the gentle case, couple geometric models have also been studied for the skew-gentle case,
cf. \cite{A21,AB,L-FSV}. In particular, Barmeier-Schroll-Wang \cite{BSW} has given three equivalent descriptions of the partially wrapped Fukaya categories of orbifold surfaces, which includes the skew-gentle case. In this paper, we will use the geometric models of graded skew-gentle algebras developed in Qiu-Zhang-Zhou (=QZZ) \cite{QZZ}.

\subsection{Geometric models for graded skew-gentle algebras}
In \cite{QZZ}, they use two geometric models.
The first one utilizes the conventional punctured marked surfaces with gradings and full formal arc systems. This approach provides a geometric interpretation for the classification of indecomposable objects in the perfect derived categories of graded skew-gentle algebras.
For the second model, QZZ proposed a novel model to understand the $\ZZ_2$-symmetry at punctures.
They replaced each puncture $P$ by a boundary component with an open marked point and a closed marked point, known as a binary $\Vot_P$. Moreover, they imposed the condition that the square of the Dehn twist along $\Vot_P$ equals the identity as shown in \Cref{fig:000} (cf. \cite[Figure 1]{QZZ}).
This geometric model is called graded marked surfaces with binary (=GMSb).
While the two models are equivalent when classifying objects (via arcs/curves),
the GMSb model provides a better interpretation for the intersection formula (cf. \Cref{thm:int formula} for its dual version).

\begin{figure}[htbp]
	\begin{tikzpicture}[scale=.3]
		\draw[ultra thick,fill=gray!10] (90:2) ellipse (1.5) node {$\Vot_P$};
    \draw[blue] (90:8) .. controls +(-45:4) and +(-30:7) .. (90:.5);
    \draw[cyan, ultra thick] (90:8) .. controls +(225:4) and +(-150:7) .. (90:.5);
    \draw[blue] (90:8)\nn;
    \draw[blue] (90:.5)\nn;
	\draw[blue]	(90:.25) node[black,below]{$m_P$};
	\draw[blue] (0,3.5)node[black,above]{$y_P$}\DEC;
	\draw (6.5,2)node{\Huge{=}};
    \begin{scope}[shift={(13,0)}]
	\draw[ultra thick,fill=gray!10] (90:2) ellipse (1.5) node {$\Vot_P$};
    \draw[blue] (90:8) .. controls +(-45:5) and +(-30:7) .. (90:.5);
    \draw[cyan, ultra thick] (90:8)
        .. controls +(-30:7) and +(-30:9) .. (0,-1)
        .. controls +(150:5) and +(160:3) .. (0,4.5)
        .. controls +(-20:4) and +(-15:2) .. (0,.5);
    \draw[blue] (90:8)\nn;
    \draw[blue] (90:.5)\nn;
	\draw[blue] (0,3.5)\DEC;
    \end{scope}	
    \end{tikzpicture}
\caption{$\Dehn{\Vot_P}^2$-action}\label{fig:000}
\end{figure}

In this paper, we will fix a GMSb $\gmsx$ and a full formal closed arc system $\ac^\ast$,
which determines a graded skew-gentle algebra $\Lambda_{\ac}$.
Conversely, for any graded skew-gentle algebra, one can reconstruct the corresponding GMSb with a full formal closed arc system. The dual version of the intersection formula by QZZ is a bijection denoted by $X$ from graded closed arcs on $\gmsx$ up to the $\mathbb{Z}_2$-symmetry to arc objects in $\mathcal{D}^b(\Lambda_{\ac})$, such that the intersection numbers between two graded closed arcs are identified with the dimensions of Hom-spaces between the corresponding two objects.

We will study finite hearts with simple tilting (cf. \cite[Definition 3.7]{KQ})
in the bounded derived category $\mathcal{D}^b(\Lambda_{\ac})$
and show that they correspond to S-graphs with flip on $\gmsx$ (cf. \cite[\S~6]{HKK} and \cite[\S~2]{CHQ}).
Recall that a mixed-angulation, consisting of open arcs, divides the surfaces into polygons such that each of which contains a closed marked point (cf. \cite{CS});
an S-graph $\CAS$ is the dual of which.
Additionally, we consider graded arcs and the closed arcs in an S-graph are required
to have positive intersection indices of intersection between them.

Our first result is that the principal component ${\rm EG}^{\circ}(\gmsx)$ is isomorphic to
the component ${\rm EG}^{\circ}\mathcal{D}^b(\sg)$ containing the canonical heart.

\begin{theorem}[\Cref{main}]
The arc-object correspondence $X_?$ in \Cref{thm:int formula} induces an isomorphism of oriented graphs
\[
\begin{array}{rcl}
X:\,\,\,\,{\rm EG}^{\circ}(\gmsx) & \to & {\rm EG}^{\circ}\mathcal{D}^b(\sg) \\
\mathbb{S}=\{\CA\} & \mapsto & \mathcal{H}_{\mathbb{S}}=\langle X_{\CA}|\,\CA\in\CAS\rangle.
   \end{array}
\]\label{main0}
\end{theorem}

\subsection{Quadratic differentials and stability conditions}
As an application, we can establish an isomorphism between the moduli space of quadratic differentials related to a GMSb $\gmsx$ and the space of stability conditions on $\mathcal D^b(\sg)$.

We introduce the moduli space of $\gmsx$-framed quadratic differentials, which is denoted as $\mathrm{FQuad}(\gmsx)$. The key difference from the usual frame is that we also have the ${\rm D}^2_{\Vot_P}$-identification. Denote by $\mathrm{Stab}(\mathcal D^b(\sg))$ the space of stability conditions on $\mathcal D^b(\sg)$, and the principal component of $\mathrm{Stab}(\mathcal D^b(\sg))$ corresponding to $\mathrm{EG}^{\circ}(\gmsx)$ is denoted by $\mathrm{Stab}^{\circ}(\mathcal D^b(\sg))$. We upgrade the isomorphism $X$ in \Cref{main0} to an injection between complex manifolds. The image of this injection is the principal component $\mathrm{Stab}^{\circ}(\mathcal D^b(\sg))$,  which is a generic-finite component corresponding to $\mathrm{EG}^{\circ}(\gmsx)$. The following theorem summarizes our main result:

\begin{theorem}[\Cref{main2}]
There exists an injection $\iota$ between the moduli spaces
\[
\mathrm{FQuad}^{\circ}(\gmsx)\to\mathrm{Stab}(\mathcal D^b(\sg)),
\]
which is extended by the isomorphism $X$ in \Cref{main0}. In particular, the image of $\iota$ is $\mathrm{Stab}^{\circ}(\mathcal D^b(\sg))$, which turns out to be the generic-finite component corresponding to $\mathrm{EG}^{\circ}(\gmsx)$.
\end{theorem}

\subsection{Contents}
This paper is organized as follows:
\begin{itemize}
    \item \Cref{pre}: we provide a review of the fundamental concepts, including hearts, simple tiltings, stability conditions, and graded skew-gentle algebras.
    \item \Cref{gmsb}: we revisit basic concepts of GMSb and the intersection formula.
    \item \Cref{S-graph}: we introduce S-graphs and their flips on a GMSb and construct an isomorphism between a principal component of the exchange graph of S-graphs and that of hearts.
    \item \Cref{quad_vs_stab}: we establish an isomorphism between the complex manifolds of quadratic differentials and stability conditions of a graded skew-gentle algebra.
\end{itemize}


\subsection*{Acknowledgement}

Lsq and Wdj would like to express their sincere gratitude to Qy for introducing them to this research area and providing invaluable guidance and supervision throughout the project.
Qy is supported by National Natural Science Foundation of China (Grant No. 12425104, No.12031007 and No.12271279) and National Key R\&D Program of China (No.2020YFA0713000).
Wdj is supported by JSPS KAKENHI KIBAN(S) 21H04994.

\section{Preliminaries}
\label{pre}
\subsection{Hearts and simple tiltings}
Recall that a \emph{heart} $\mathcal{H}$ in a triangulated category $\mathcal{D}$ is defined as an additive subcategory which satisfies (cf. \cite{B1}):
\begin{itemize}
\item ${\rm Hom}_{\mathcal{D}}(\mathcal{H}[m],\mathcal{H})=0$ for any positive integer $m$;
\item for each object $E$ in $\mathcal{D}$, there exists a canonical filtration (called \emph{Harder-Narasimhan filtration})
\begin{equation}
\begin{tikzcd}
0 = E_0 \ar[rr]   &&  E_1 \ar[dl] 
 \ar[r] & \cdots  \ar[r] & E_{m-1} \ar[rr] && E_m \ar[dl] = E,\\
& A_1 \ar[ul,dashed] &&&& A_m \ar[ul,dashed]
\end{tikzcd}\label{HN-filt}
\end{equation}
where $A_i \in \mathcal H[k_i]$ and $k_1 > \cdots > k_m$ are integers.
\end{itemize}

The $k$-th homology of an object $E$ in $\mathcal{D}$ with respect to the heart $\mathcal{H}$ is defined using Harder-Narasimhan filtration as:
\[
 H_k(E)=
 \begin{cases}
   A_i[-k_i], & \text{if $k=k_i$,} \\
   0, & \text{otherwise.}
 \end{cases}
\]

It is worth mentioning that a heart of a triangulated category is an abelian category (cf. \cite{BBD}).

Recall that a \emph{torsion pair} in an abelian category $\mathcal{C}$ is defined as an ordered pair of full subcategories $\langle \mathcal{T}, \mathcal{F} \rangle$ satisfying:
\begin{itemize}
    \item $\Hom_{\hua{C}}(\hua{T},\hua{F})=0$;
    \item for every object $E$ in $\hua{C}$, there exists a short exact sequence
$0 \to T \to E \to F \to 0$, where $T \in \hua{T}$ and $F \in \hua{F}$.
\end{itemize}

For a triangulated category $\mathcal{D}$ with a heart $\mathcal{H}$, if $\langle\hua{T},\hua{F}\rangle$ is a torsion pair in $\h$, then the \emph{forward tilt} of $\mathcal{H}$ with respect to $\langle \mathcal{T}, \mathcal{F} \rangle$ is defined as
\[
    \h^\sharp
    :=\{ E \in \hua{D}\,|\,H_0(E) \in \hua{T}, H_1(E) \in \hua{F}
        \mbox{ and } H_{i}(E)=0 \mbox{ otherwise} \}.
\]
The \emph{backward tilt} $\mathcal{H}^\flat$ of $\mathcal{H}$ with respect to $\langle \mathcal{T}, \mathcal{F} \rangle$ is defined dually. These two full subcategories $\mathcal{H}^\sharp$ and $\mathcal{H}^\flat$ are both hearts in $\hua{D}$ (cf. \cite{HRS}). It is worth noting that $\mathcal{H}^\flat = \mathcal{H}^\sharp[-1]$.

In particular, a forward tilt of a heart $\h$ is called \emph{simple} if the torsion-free part $\hua{F}$ is generated by a single rigid simple $S$ (i.e. $\Ext^1_{\mathcal{D}}(S,S)=0$). The new heart is denoted by $\tilt{\h}{\sharp}{S}$. Similarly, a backward tilt of $\h$ is called simple if the torsion part $\hua{T}$ is generated by a rigid simple $S$, and the new heart is denoted by $\tilt{\h}{\flat}{S}$ (cf. \cite{KQ}).

For a heart $\mathcal H$, we denote by
${\rm Sim}\,\mathcal H$ a complete set of non-isomorphic simples in $\mathcal H$. A heart $\mathcal H$ is called \emph{finite} if ${\rm Sim}\,\mathcal{H}$ is a finite set that generates $\mathcal{H}$ through extensions, or equivalently, every object $M$ in $\mathcal{H}$ has a finite filtration with simple factors.

The \emph{total exchange graph} $\EG(\D)$ of a triangulated category $\D$ is defined as a oriented graph, whose vertices are all hearts in $\D$ and edges correspond to simple forward tiltings between these hearts. The \emph{principal component} $\EG^{\circ}(\mathcal{D})$ is defined as the connected component of $\EG(\mathcal{D})$ that contains the canonical heart (cf. \cite{KQ}).

\subsection{Stability conditions}
Let's review Bridgeland's concept of a stability condition on a triangulated category (cf. \cite{B1}).

A \emph{stability condition} $\sigma=(Z, \mathcal P)$ on a triangulated category $\mathcal D$ consists of a group homomorphism $Z: K(\mathcal D)\to \mathbb C$, called the \emph{central charge}, and full additive subcategories $\mathcal P(\phi)\subset \mathcal D$ for each $\phi\in \mathbb R$, known as the \emph{slicings}. They satisfy:

\begin{enumerate}
\item[(a)] if $0\ne E\in \mathcal P(\phi)$, then $Z(E) \in \mathbb R_{>0}\cdot{\rm exp}({\rm i}\pi\phi)$;
\item[(b)] for all $\phi\in \mathbb R$, $\mathcal P(\phi +1) = \mathcal P(\phi)[1]$;
\item[(c)] if $\phi_1>\phi_2$ and $A_i\in \mathcal P(\phi_i)\,(i=1,2)$, then ${\rm Hom}_{\mathcal D}(A_1, A_2) = 0$;
\item[(d)] for any $0\neq E\in \mathcal D$, there is a finite sequence of real numbers
\[
\phi_1>\phi_2>\dots>\phi_m
\]
and Harder-Narasimhan filtration as \Cref{HN-filt} with $A_i\in\mathcal P(\phi_i)$ for all $1\leq i\leq m$.
\end{enumerate}

An object $E\in \mathcal P(\phi)$ for some $\phi\in \mathbb R$ is called \emph{semistable}. Furthermore, if $E$ is simple in $\mathcal P(\phi)$, then it is called \emph{stable}.

We will only consider stability conditions that satisfy the \emph{support property} (cf. \cite{KS}), and the set of all stability conditions on $\mathcal D$ that satisfy this support property is denoted by $\mathrm{Stab}(\mathcal D)$. The set $\mathrm{Stab}(\mathcal D)$ has a natural complex manifold structure with a local coordinate
\[
    Z\in \mathrm{Hom}_{\mathbb Z}(K(\mathcal D),\mathbb C).
\]

There is an alternative definition of stability conditions. A \emph{stability function} on an abelian category $\mathcal C$ is a group homomorphism $Z : K_0(\mathcal C)\to \mathbb C$ that satisfies: for each nonzero object $E\in\mathcal C$, the complex number $Z(E)$ lies in the semi-closed upper half plane
\[
\mathbb H_{+}=\{r\mathrm{exp}({\rm i}\pi\phi)\,|\,r>0,\,0<\phi\le1\}\subset\mathbb C,
\]
where $K_0(\mathcal C)$ is the Grothendieck group of $\mathcal C$. According to Bridgeland \cite[Proposition 5.3]{B1}, a stability condition can also be defined on a triangulated category $\mathcal D$ in terms of a pair $(\mathcal H,Z)$, where $\mathcal H$ is a heart of $\mathcal D$ and $Z$ is a stability condition on $\mathcal H$ that satisfies the Harder-Narasimhan property.

The set $\rm{Stab}(\mathcal D)$ has a natural $\mathbb C$-action: for any $s\in\mathbb{C}$, we have
\[
s \cdot (Z, \mathcal P) = (Z\cdot {\rm exp}({\rm i}\pi s), \mathcal P_{-\mathrm{Re}(s)}),
\]
where $\mathcal P_{-\mathrm{Re}(s)}(m) = \mathcal P(m-\mathrm{Re}(s))$ for any $m\in\mathbb R$. A connected component $\mathrm{Stab}_0(\mathcal D)$ of $\mathrm{Stab}(\mathcal D)$ is called  \emph{generic-finite} if there exists a connected component $\mathrm{EG}_0(\mathcal D)$ of $\mathrm{EG}(\mathcal D)$, whose vertices are all finite hearts, such that
\[
\mathrm{Stab}_0(\mathcal D)=\mathbb C\cdot\bigcup_{\mathcal H\in\mathrm{EG}_0(\mathcal D)}\mathrm{U}(\mathcal H),
\]
where $\mathrm{U}(\mathcal H)$ is the subspace in $\mathrm{Stab}(\mathcal D)$ consisting of those stability conditions $\sigma=(\mathcal H, Z)$ whose central charge $Z$ takes values in $\mathbb H=\{z\in\mathbb C\ |\ \mathrm{Im}(z)>0\}$.

\subsection{Graded skew-gentle algebras}
Let's review the concept of a graded skew-gentle algebra (cf. \cite{QZZ}). From now on, $\k$ is a fixed algebraically closed field.

A \emph{graded quiver} is a quiver $Q=(Q_0, Q_1, s, t)$, where $Q_0$ is the set of vertices, $Q_1$ is the set of arrows, and $s, t: Q_1 \rightarrow Q_0$ are the start and terminal functions respectively, equipped with a grading, which is a function $|\cdot|: Q_1 \rightarrow \mathbb{Z}$.

A pair $(Q, I)$ of a graded quiver $Q$ and a relation set $I$ is called a \emph{gentle pair} if the algebra $\k Q/I$ is a \emph{gentle algebra}, or equivalently, they satisfy:
\begin{itemize}
    \item every element in $I$ is a path of length 2;
    \item each vertex in $Q_0$ has at most two incoming arrows and at most two outgoing arrows;
    \item for any arrow $\alpha$ in $Q_1$, there is at most one arrow $\beta$ (resp. $\gamma$) such that $\alpha\beta\in I$ (resp. $\gamma\alpha\in I$);
    \item for any arrow $\alpha$ in $Q_1$, there is at most one arrow $\beta$ (resp. $\gamma$) such that $\alpha\beta\notin I$ (resp. $\gamma\alpha\notin I$).
\end{itemize}

A triple $(Q,\Sp,I)$, consisting of a graded quiver $Q$, a subset $\Sp$ of $Q_0$, and a relation set $I$, is called a \emph{graded skew-gentle triple} if the pair $(Q^{\mathrm{sp}},I^{\mathrm{sp}})$ is a gentle pair, where
    \begin{itemize}
        \item $Q_0^{\mathrm{sp}}=Q_0$,
        \item $Q_1^{\mathrm{sp}}=Q_1\cup\{\varepsilon_i\mid i\in \Sp\}$, with $s(\varepsilon_i)=t(\varepsilon_i)=i$ and $|\varepsilon_i|=0$,
        \item $I^{\mathrm{sp}}=I\cup\{\varepsilon_i^2\mid i\in \Sp \}$.
    \end{itemize}
    A finite-dimensional graded algebra $\Lambda$ is called a \emph{graded skew-gentle algebra} if it is Morita equivalent to
   $$\Lambda(Q,\Sp,I):=\k Q^{\mathrm{sp}}/\<I\cup\{\varepsilon_i^2-\varepsilon_i\mid i\in \Sp\}\>,$$
    where $(Q,\Sp,I)$ is a graded skew-gentle triple. In particular, if $\Sp$ is empty, then the graded algebra $\Lambda$ is gentle.

\begin{example}
Let $Q$ be the following graded $A_3$ quiver
\[
\begin{tikzcd}
  1 \ar[r,"a"] & 2 \ar[r,"b"] & 3
\end{tikzcd}
\]
with $\Sp=\{3\}$ and $I=\varnothing$. Then the quiver $Q^{\mathrm{sp}}$ is the following quiver
\[
\begin{tikzcd}
  1 \ar[r,"a"] & 2 \ar[r,"b"] & 3 \ar[loop right, distance=2em, "\varepsilon"]
\end{tikzcd}
\]
with $|\epsilon|=0$ and $I^{\mathrm{st}}=\{\varepsilon^2\}$. $(Q^{\mathrm{sp}},I^{\mathrm{sp}})$ is a gentle pair, and now we get a graded skew-gentle algebra
\[\k Q^{\mathrm{sp}}/\<\varepsilon^2-\varepsilon\>,\]
which is isomorphic to the path algebra of the classical $D_4$ quiver. The quiver $Q^{\mathrm{sp}}$ with relation $\varepsilon^2-\varepsilon$ is called the graded skew-gentle $D_4$ quiver. Similarly, we can also construct the graded skew-gentle $\widetilde{D}_4$ quiver from graded $A_3$ quiver by adding two idempotent loops at vertices $1$ and $3$.
\end{example}

\section{Graded marked surfaces with binary}
\label{gmsb}
In this section, we review some concepts of graded marked surfaces with binary and the intersection formula ${\rm Int}\,=\,{\rm dim}\,{\rm Hom}$ following \cite{QZZ}.
\subsection{Basic notions of GMSb}
Recall that a \emph{graded marked surface with binary} (GMSb) $\gmsx=(\surf,\M,\Y,\vot,\lambda)$ is a compact connected oriented surface $\surf$ with a non-empty boundary $\partial \surf$, equipped with the following parts:
\begin{itemize}
\item a finite set $\M\subset\partial \surf$ of \emph{open marked points} and a finite set $\Y\subset\partial\surf$ of \emph{closed marked points}. In each component of $\partial \surf$, marked points in $\M$ and $\Y$ are alternative;
\item a finite set $\vot$ consists of \emph{binaries}, where each binary $\Vot_P$ is a boundary component containing one open marked point $m_P$ and one closed marked point $y_P$ on it;
\item a \emph{grading} $\grad$ on $\surf$, represented by a class in $H^1\left(\mathbb{P}T\surfi,\ZZ\right)$, where $\surfi=\surf\backslash\partial\surf$, and $\mathbb{P}T\surfi$ is the projectivized tangent bundle. It satisfies:
\begin{itemize}
    \item $\grad$ takes a value of 1 on each clockwise loop ${p}\times\mathbb{R}\mathbb{P}^1$ on $\mathbb{P}T_p\surfi$;
    \item $\grad$ takes a value of $1$ on each clockwise loop $l_P\times\{x\}$ on $\surf$ around any binary $\Vot_P$, where $x\in\mathbb{RP}^1$, or equivalently,
\[
    \grad[ l_P\times\{x\} ]=1.
\]
\end{itemize}
Note that each grading corresponds to a section of $\mathbb{P}T\surfi$.
\end{itemize}

An \emph{isomorphism of graded surfaces} $\surf^{\lambda}\to\underline\surf^{\lambda'}$ is a pair $(f,h)$, where $f:\surf\to\underline{\surf}$ is an orientation-preserving local diffeomorphism and $h$ is a homotopy class of paths in the space of sections of $\mathbb PT\surf^{\circ}$ from $f^{\ast}{\lambda'}$ to $\lambda$. An \emph{isomorphism of graded marked surfaces with binary} $\gmsx\to\underline\surf^{\lambda'}_{\vot'}$ is an isomorphism of graded surfaces $(f,h):\surf^{\lambda}\to\underline\surf^{\lambda'}$ with $f(\vot)=\vot'$ and $f(\M)\to\M'$.

Recall that a \emph{curve} on $\gmsx$ is an immersion $c:I\to \surf$ where $I=[0,1]$ or $S^1$. A \emph{grading} $\wc$ on $c$ is a homotopy class of paths in $\mathbb{P}T_{c(t)}\surfi$ from $\grad(c(t))$ to $\dot{c}(t)$, which varies continuously with $t\in I$. The pair $(c,\wc)$ (or $\wc$ for short) is called a \emph{graded curve}.

For any graded curve $\wc$ and any $\rho\in\ZZ$, $\wc[\rho]$ is defined as the graded curve with the same underlying curve as $\wc$, but with a grading obtained by composing $\wc(t):\grad(c(t))\to\dot{c}(t)$ with the path from $\dot{c}(t)$ to itself, given by clockwise rotation by $\pi$.

For any two graded curves $(c_1,\wc_1), (c_2,\wc_2)$ on $\gmsx$,
let $q=c_1(t_1)=c_2(t_2)\in \surfi$ be a point where $c_1$ and $c_2$ intersect transversely.
The \emph{intersection index} of $\wc_1$ and $\wc_2$ at $q$ is defined as
\[{\rm ind}_q(\wc_1,\wc_2)=\wc_1(t_1)\cdot\kappa\cdot\wc_2^{-1}(t_2)\ \in\pi_1(\mathbb{P}T_{q}\surf^\circ)\cong \ZZ,\]
where $\kappa$ represents (the homotopy class of) the path in $\mathbb{P}T_q\surfi$ from $\dot{c}_1(t_1)$ to $\dot{c}_2(t_2)$
given by a clockwise rotation by an angle smaller than $\pi$.

Moreover, for the case where $q\in\partial \surf$, let $l\subset\surfi$ be a small half-circle centered at $q$. We consider an embedded arc $\alpha:[0,1]\to l$ that moves clockwise around $q$, intersecting $c_1$ and $c_2$ at $\alpha(0)$ and $\alpha(1)$, respectively. The choice of $\alpha$ is unique up to a change of parametrization. By fixing an arbitrary grading $\wa$ on $\alpha$, the intersection index ${\rm ind}q(\wc_1,\wc_2)$ is defined as follows:
\[
{\rm ind}_q(\wc_1,\wc_2):={\rm ind}_{\alpha(0)}(\wc_1,\wa)-{\rm ind}_{\alpha(1)}(\wc_2,\wa).
\]
For intersections between graded curves $\wc_1$ and $\wc_2$ and for any $\rho\in\ZZ$, we have
\[
    {\rm ind}_q(\wc_1,\wc_2[\rho])={\rm ind}_q(\wc_1,\wc_2)-\rho.
\]

Let $\wc_1,\wc_2$ be two graded curves on $\gmsx$ in a minimal position (i.e. choose $\wc_1,\wc_2$ in their homotopy class respectively such that the number of their intersections is minimal). An intersection $q$ between $\wc_1$ and $\wc_2$ is called an \emph{oriented intersection from $\wc_1$ to $\wc_2$}, meaning there is a small arc in $\surfi$ around $q$ from a point in $\wc_1$ to a point in $\wc_2$ in a clockwise direction. This set of oriented intersections from $\wc_1$ to $\wc_2$ is denoted by $\overrightarrow{\cap}(\wc_1,\wc_2)$.
Similarly we have the following subsets of oriented intersections:
\begin{itemize}
    \item $\overrightarrow{\cap}^{\rho}(\wc_1,\wc_2)$: the set of oriented intersections from $\wc_1$ to $\wc_2$ with index $\rho$;
    \item $\overrightarrow{\cap}(\wc_1,\wc_2)$: the set of oriented intersections from $\wc_1$ to $\wc_2$;
    \item $\overrightarrow{\cap}_{\surfi}^{\rho}(\wc_1,\wc_2)$: the subset of $\overrightarrow{\cap}^{\rho}(\wc_1,\wc_2)$ consisting of oriented intersections in $\surfi$;
    \item $\overrightarrow{\cap}_{\surfi}(\wc_1,\wc_2)$: the subset of $\overrightarrow{\cap}(\wc_1,\wc_2)$ consisting of oriented intersections in $\surfi$.
\end{itemize}
An oriented intersection is also called an angle. In fact, the intersection index is defined for angles $a\in \overrightarrow{\cap}(\wc_1,\wc_2)$. We denote the intersection index of an angle $a$ by ${\rm ind}(a)$. For three angles $a$, $b$ and $c$ at the same point $q$, we say that $c=a+b$ as oriented intersections if the angles $a$ and $b$ are composable and the composition of these two angles is the angle $c$ (cf. \Cref{fig:indices}).

An \emph{open} (resp. \emph{closed}) \emph{curve} on $\gmsx$ is a curve $\gamma:I\rightarrow \surf$, satisfying
	\begin{itemize}
		\item $\gamma(\partial I)\subset \M$ (resp. $\gamma(\partial I)\subset \Y$), and $\gamma(I\setminus\partial I)\subset\surfi$;
		\item $\gamma$ is not homotopic to a point.
	\end{itemize}
An open (resp. closed) curve is called an open (resp. closed) \emph{arc} if $I=[0,1]$. Open (resp. closed) curves are considered up to homotopy relative to $\partial I$.

We always consider \emph{admissible arcs} (cf. \cite[Definition 2.20]{QZZ}), but there are also many examples of non-admissible arcs on $\gmsx$ (cf. \cite[Figure 22]{QZZ}). The set of graded admissible closed arcs on $\gmsx$ is denoted by $\wCA(\gmsx)$.

Let $\Dfang$ be the subgroup of the usual mapping class group of $\gmsx$ generated by the squares of Dehn twists $\mathrm{D}_{\Vot}$ along any binary $\Vot\in\vot$.
The binary mapping class group $\MCG(\gmsx)$ of $\gmsx$ is defined as the quotient of the usual mapping class group by $\Dfang$.

\begin{definition}\label{twin}
Two graded admissible closed arcs $\CA$ and $\CA'$ are said to be \emph{twins} if $\CA'={\rm D}_{\Vot}(\CA)[n]$ for some $\Vot\in\vot$ and some integer $n$. Each of them is called a twin of the other. Moreover, if $n=0$, they are called \emph{identical twins}. Otherwise, they are called \emph{fraternal twins}.
\end{definition}

\begin{remark}
If $\CA$ and $\CA'$ are twins, the angle $a\in\overrightarrow{\cap}(\CA,\CA')$ is on the binary, and $b\in\overrightarrow{\cap}(\CA,\CA')$ is not on the binary, we can calculate that ${\rm ind}(a)={\rm ind}(b)$ by the definition of intersection indices and the local grading around a binary (cf. \cite[Figure 6]{QZZ}). In fact, this is equivalent to the fact that the winding number around a binary is one (cf. \cite[Remark 1.14]{QZZ}). \label{winding_binary}
\end{remark}

The \emph{$\Dfang$-orbit} $\Dfang\cdot\ws$ of a graded admissible closed arc $\ws\in\wCA(\gmsx)$ consists of the graded admissible closed arcs that are obtained from $\ws$ by actions of $\Dfang$ on the ends (which are in binaries) separately. Consider the cyan arcs in the two pictures of \Cref{fig:nomaps} (cf. \cite[Figure 23]{QZZ}). These arcs are in the same $\Dfang$-orbit.

\begin{figure}[htbp]
	\begin{tikzpicture}[scale=.35]
	\draw[ultra thick]plot [smooth,tension=1] coordinates {(3,-4.5) (-90:4) (-3,-4.5)};
	\draw[ultra thick,fill=gray!10] (90:2) ellipse (1.5) node {$\Vot_P$};
    \draw[blue, thick] (90:-4) .. controls +(45:5) and +(30:4) .. (90:3.5);
    \draw[blue, thick] (90:-4) .. controls +(135:5) and +(150:4) .. (90:3.5);
    \draw[red] (90:8) .. controls +(-45:4) and +(-30:7) .. (90:.5);
    \draw[cyan, ultra thick] (90:8) .. controls +(225:4) and +(-150:7) .. (90:.5);
    \draw (90:8)\DEC;
	\draw (90:.5)\DEC;
	\draw[blue]	(90:.25) node[black,below]{$y_P$};
	\draw[blue] (0,3.5)node[black,above]{$m_P$}\nn;
	\draw[blue] (-90:4)\nn;
    \begin{scope}[shift={(13,0)}]
	\draw[ultra thick]plot [smooth,tension=1] coordinates {(3,-4.5) (-90:4) (-3,-4.5)};
	\draw[ultra thick,fill=gray!10] (90:2) ellipse (1.5) node {$\Vot_P$};
    \draw[blue,thick] (90:-4) .. controls +(45:5) and +(30:4) .. (90:3.5);
    \draw[blue,thick] (90:-4) .. controls +(135:5) and +(150:4) .. (90:3.5);
    \draw[red] (90:8) .. controls +(-45:5) and +(-30:7) .. (90:.5);
    \draw[cyan, ultra thick] (90:8)
        .. controls +(-30:7) and +(-30:9) .. (0,-1)
        .. controls +(150:5) and +(160:3) .. (0,4.5)
        .. controls +(-20:4) and +(-15:2) .. (0,.5);
    \draw (90:8)\DEC;
	\draw (90:.5)\DEC;
	\draw[blue] (0,3.5)\nn;
	\draw[blue] (-90:4)\nn ;
    \end{scope}	
    \end{tikzpicture}
\caption{$\Dehn{\Vot_P}^2$-action on the cyan arc}\label{fig:nomaps}
\end{figure}

\subsection{The intersection formula}
From now on, we always consider $\wCA(\gmsx)/\Dfang$ rather than $\wCA(\gmsx)$. We call $\wCA(\gmsx)/\Dfang$ the set of graded unknotted (closed) arcs on $\gmsx$.

Let $\gmsx$ be a GMSb. A collection $\mathbf T$ of graded open (resp. closed) arcs is called
\begin{itemize}
\item an \emph{open} (resp. \emph{closed}) \emph{arc system}, if $|\overrightarrow{\cap}_{\surfi}(\wc_1,\wc_2)|=0$ for any $\wc_1,\wc_2\in\mathbf T$;
\item a \emph{full formal} \emph{open} (resp. \emph{closed}) \emph{arc system}, if $\mathbf T$ is an open (resp. closed) arc system which cuts out $\surf$ into polygons (by choosing a representative in the $\Dfang$-orbit for each arc in $\mathbf T$ pairwisely in a minimal position), called \emph{$\mathbf T$-polygons}, such that each $\mathbf T$-polygon contains exactly one point in $\Y$ (resp. $\M$).
\end{itemize}

In fact, every GMSb $\gmsx$ admits a full formal open arc system and its dual full formal closed arc system, cf. \cite[Lemma 1.20 and Construction 4.1]{QZZ}.

Now we fix a GMSb $\gmsx$ with a full formal closed arc system $\ac^\ast$. Then there is a graded skew-gentle algebra $\sg$ with a graded skew-gentle triple $(Q,{\rm Sp},I)$ obtained from the dual $\ac$ of $\ac^{\ast}$, which satisfies the following conditions (cf. \cite[Remark 1.24]{QZZ}):

\begin{itemize}
    \item each vertex $i\in Q_0\backslash{\rm Sp}$ corresponds to an open arc $\widetilde{\gamma}_i$ which has no twin in $\mathbb{A}$;
    \item there is an arrow $i\to j$ in $Q_1$ with degree $d$ if and only if there is an $\mathbb{A}$-polygon which does not enclose a binary, having $\widetilde{\gamma}_i$ and $\widetilde{\gamma}_j$ as consecutive edges meeting at $p$, where $\widetilde{\gamma}_j$ follows $\widetilde{\gamma}_i$ in the clockwise order, with index ${\rm ind}_p(\widetilde{\gamma}_j,\widetilde{\gamma}_i)=d$;
    \item each vertex $i\in {\rm Sp}$ corresponds to twins ${\widetilde{\gamma}_i}$ and $\widetilde{\gamma}'_i$ in $\mathbb{A}$;
    \item $I$ consists of $a_1a_2$ for $a_1:i\to j$ and $a_2:j\to l$ if and only if $i,j,l$ are consecutive edges in an $\mathbb{A}$-polygon which does not enclose a binary.
\end{itemize}

We say that $\sg$ obtained from $\ac$ above is the graded skew-gentle algebra associated to $\mathbb{A}^{\ast}$.

In fact, for any graded skew-gentle algebra $\Lambda$, there exists a GMSb $\gmsx$ with a full formal closed arc system $\ac^{\ast}$ such that $\Lambda$ is the graded skew-gentle algebra associated to $\ac^{\ast}$ (cf. \cite[Theorem 1.23]{QZZ}).

Let $\ws$ and $\wt$ be two graded closed arcs in $\wCA(\gmsx)/\Dfang$. The intersection number from $\ws$ to $\wt$ of index $\rho$ is defined as
$$\oIntd(\ws,\wt):=\min\{|\overrightarrow{\cap}^\rho(\ws',\wt')|\mid\ws'\in\Dfang\cdot\ws,\ \wt'\in\Dfang\cdot\wt\},$$
where $\overrightarrow{\cap}^\rho(\ws',\wt')$ consists of the clockwise angles at intersections from $\ws'$ to $\wt'$.

Then we state the correspondence between graded admissible closed arcs on $\gmsx$ and the objects in $\mathcal{D}^b(\sg)$. 

Recall that a graded surface has a double cover $\tau$ given by the orientations of the foliation lines in the natural full formal arc system and $\mathbb{Z}_{\tau}=\mathbb{Z}\otimes_{\mathbb{Z}/2}\tau$. As a corollary of the intersection formula (cf. \cite[Theorem 4.11]{QZZ}), we have the following group isomorphism:

\begin{lemma}
There is a natural isomorphism of abelian groups
$$\xi:H_1(\mathbf{S},\mathbf{Y};\mathbb{Z}_{\tau})\cong K_0(\mathcal{D}^b(\sg))$$
such that $\xi$ sends the initial full formal closed arc system $\ac^{\ast}=\{\CA_i\}_{i=1}^n$ on $\gmsx$ to a complete set $\{S_i=\xi(\CA_i)\}_{i=1}^n$ of simple objects in $\mathrm{mod}\,\sg$, where $H_1(\mathbf{S},\mathbf{Y};\mathbb{Z}_{\tau})$ is the homology group with coefficients in $\mathbb{Z}_{\tau}$, and $K_0(\mathcal{D}^b(\sg))$ is the Grothecdieck group of $\mathcal{D}^b(\sg)$.
\label{K0 vs H1}
\end{lemma}

Since $\ac^{\ast}$ is a basis of the abelian group $H_1(\gmsx,\mathbf{Y};\mathbb{Z}_{\tau})$, we only need to define $\xi$ on $\ac^{\ast}$, and it will be defined linearly on the whole group. For convenience, we denote $H_1(\gmsx,\mathbf{Y};\mathbb{Z}_{\tau})$ by $H_1(\gmsx)$ in this paper.

Let $\good$ be the additive subcategory consisting of all (isoclasses of) \xian objects in $\mathcal{D}^b(\sg)$ (for the initial definition of arc objects, see \cite[Section 3]{QZZ}). Then on $\good$, Lemma \ref{K0 vs H1} can be upgraded to the following intersection formula, which is the dual version of \cite[Theorem 4.11]{QZZ}.

\begin{theorem}\label{thm:int formula}
There is a bijection
\[
X:\wCA(\gmsx)/\Dfang\to\good
\]
such that for any graded closed arcs $\ws$ and $\wt$ in $\wCA(\gmsx)/\Dfang$, we have
\[
\oIntd(\ws,\wt)=\dim\Hom_{\mathcal{D}^b(\sg)}(X_{\ws},X_{\wt}[\rho]).
\]
\end{theorem}

That is, we have the relationship $\xi[\ws]=X_{\ws}$ if $X_{\ws}\in\good$. For convenience, we denote $X^{-1}(T)$ by $\CA_T$ for any $T\in\good$.
\section{S-graphs on graded marked surfaces with binary}
\label{S-graph}
In this section, we fix a GMSb $\gmsx$ with a full formal closed arc system $\ac^{\ast}$. Let $n$ be the number of closed arcs in $\ac^\ast$. Let $\sg$ be the associated skew-gentle algebra.
\subsection{Intersection indices}
In this paper, we always discuss the oriented intersections between two graded closed arcs $\wa$ and $\wb$ on $\gmsx$. But some of them are counted in the intersection numbers $\oIntd$, and others are not counted.

\begin{convention}
Let $\wa$ and $\wb$ be two graded closed arcs on $\gmsx$. After fixing representatives of them, for an oriented intersection $a\in\overrightarrow{\cap}^{\rho}(\wa,\wb)$, if $a$ is counted in $\oIntd(\wa,\wb)$, we call $a$ a \emph{genuine-intersection}. Otherwise, we call $a$ a \emph{pseudo-intersection}. Moreover, from now on, the concept ``intersection indices'' means the indices of genuine-intersections.\label{pseudo-int}
\end{convention}

\begin{lemma}
For a full formal closed arc system $\mathbb{S}$ on $\gmsx$ and $\wa,\wb\in\mathbb{S}$, if an oriented intersection $a\in\overrightarrow{\cap}(\wa,\wb)$ is pseudo, then $\wa$ and $\wb$ are twins.\label{lem:twin}
\end{lemma}

\begin{proof}
By \cite[Lemma 5.8]{QZZ}, we have $\oIntd(\wa,\wb)=\oIntd(\wa^{\times},\wb^{\times})$, where $\wa^{\times},\wb^{\times}$ are the corresponding tagged arcs in $\mathbf{S}^{\lambda}$ given by \cite[Lemma 5.4]{QZZ}, and $\oIntd(\wa^{\times},\wb^{\times})$ is the number of tagged oriented intersection (TOI) from $\wa^{\times}$ to $\wb^{\times}$. Let $a^{\times}$ be the angle from $\wa^{\times}$ to $\wb^{\times}$ which corresponds to $a$ when passing from GMSp to GMSb. Since $a$ is a pseudo-intersection, $a^{\times}$ does not count a TOI. By \cite[Remark 5.7]{QZZ}, there are cases (ii) to (vii) of oriented intersections that do not count TOI. If $a^{\times}$ belongs to case (ii) or (vii), then $\wa$ and $\wb$ are twins as \cite[Figure 30 and 35]{QZZ}. Otherwise, if $a^{\times}$ belongs to one of cases (iii)-(vi), there is no corresponding oriented intersection $a$, which is a contradiction.
\end{proof}

Then we discuss some properties of intersection indices.

\begin{definition}
Let $\wa$ and $\wb$ be graded closed arcs on $\gmsx$, and $\ang$ be an angle from $\wa$ to $\wb$ at their common endpoint $\wa(0)=\wb(0)$. We define $\wg=\wa\wedge_a\wb$ to be the \emph{smoothing out} of $\wa\cup\wb$ at $\wa(0)=\wb(0)$ along $a$, connecting $\wa(1)$ and $\wb(1)$, cf. \Cref{fig:008}. In particular, we require that the grading of $\wg$ is inherited from $\wa$.

Moreover, if $\wa, \wb$ and $\wg$ all have no self-intersection in $\mathbf{S}^{\circ}$, we say that $\wa$, $\wb$ and $\wg$ form a \emph{contractible triangle} on $\gmsx$. \label{smoothing}
\end{definition}

\begin{figure}[htbp]
	\begin{tikzpicture}[scale=.35]
	\draw[ultra thick]plot [smooth,tension=1] coordinates {(3,-4.5) (-90:5.2) (-3,-4.5)};
	\draw[ultra thick]plot [smooth,tension=1] coordinates {(-6,-9) (-5,-10.5) (-4.3,-12.5)};
	\draw[ultra thick]plot [smooth,tension=1] coordinates {(6,-9) (5,-10.5) (4.3,-12.5)};
	\draw[red] (-90:5.2) .. controls +(-120:3) and +(30:3) .. (-5,-10.5);
	\draw[red] (-90:5.2) .. controls +(-60:3) and +(150:3) .. (5,-10.5);
	\draw[cyan, ultra thick] (-5,-10.5) .. controls +(20:3) and +(160:3) .. (5,-10.5);
	\draw (0,-6.7) node {$a$} ;
    \draw (-2.3,-9.3) node {$b$} ;
    \draw (2.3,-9.4) node {$c$} ;
	\draw[Green,->, thick,>=stealth]plot
	(0.5,-6.1)to[bend left=25](-0.5,-6.1);
	\draw[Green,->, thick,>=stealth]plot
	(2.9,-10)to[bend left=25](3.1,-9.1);
	\draw[Green,->, thick,>=stealth]plot
	(-3.1,-9.1)to[bend left=25](-2.9,-10);
	\draw (-90:5.2)\DEC;
	\draw (-5,-10.5)\DEC;
	\draw (5,-10.5)\DEC;
    \draw (-90:10.8) node {$\wg$} ;
    \draw (2.5,-7.5) node {$\wa$} ;
    \draw (-2.5,-7.5) node {$\wb$} ;
    \draw (0,-4.3) node {$A$} ;
    \draw (-5.8,-10.5) node {$B$} ;
    \draw (5.8,-10.5) node {$C$} ;
	\end{tikzpicture}
	\caption{Contractible triangle}\label{fig:008}
    \end{figure}

\begin{lemma}
Let $\wa$, $\wb$ and $\wg$ be graded closed arcs on $\gmsx$.
\numbers
    \item (cf. \cite[Lemma 3.10]{FQ}) If in addition, $\wa(0)=\wb(0)=\wg(0)=M\in\partial\mathbf{S}$, assuming that $\wa$, $\wb$, $\wg$ are clockwise around $p$, and three angles $a\in\overrightarrow{\cap}(\wa,\wb), b\in\overrightarrow{\cap}(\wb,\wg)$ and $c\in\overrightarrow{\cap}(\wa,\wg)$ at $M$ satisfy $c=a+b$ as oriented intersections, cf. \Cref{fig:indices}, then
    \[
        {\rm ind}(c)={\rm ind}(a)+{\rm ind}(b).
    \]

    \begin{figure}[htbp]\centering
	\def\dexc{black!10!blue!50!green}
	\begin{tikzpicture}[scale=0.9]
	\draw[ultra thick]plot [smooth,tension=1] coordinates {(-2,0.2) (0,0) (2,0.2)};
	\draw[red] (0,0) .. controls +(-120:1) and +(30:1) .. (-2,-2.5);
	\draw[red] (0,0) .. controls +(-80:0.5) and +(100:0.5) .. (0.5,-3);
	\draw[red] (0,0) .. controls +(-70:0.5) and +(160:0.5) .. (3,-2.5);
    \draw(0,0)\DEC;
    \draw(-2,-2.5)\DEC;
    \draw(0.5,-3)\DEC;
    \draw(3,-2.5)\DEC;
    \draw (-1.4,-1.6)node{$\wg$};
    \draw (0.1,-2.3)node{$\wb$};
    \draw (1.65,-1.2)node{$\wa$};
    \draw (-0.1,-0.7)node{$c$};
    \draw (0.5,-1.2)node{$a$};
    \draw (-0.3,-1.3)node{$b$};
    \draw (0,0.3)node{$M$};
    \draw[Green,->, thick,>=stealth]plot
	(-57:0.5)to[bend left=25](-120:0.5);
	\draw[Green,->, thick,>=stealth]plot
	(-52:1)to[bend left=25](-81:1);
	\draw[Green,->, thick,>=stealth]plot
	(-82:1)to[bend left=25](-122:1);
    \end{tikzpicture}
    \caption{Property $1^{\circ}$ of intersection indices}\label{fig:indices}
\end{figure}

    \item (cf. \cite[Proposition 3.12]{FQ}) If in addition, $\wa$, $\wb$ and $\wg$ form a contractible triangle on $\gmsx$ with interior angles $a\in\overrightarrow{\cap}(\wa,\wb), b\in\overrightarrow{\cap}(\wb,\wg)$ and $c\in\overrightarrow{\cap}(\wg,\wa)$, cf. \Cref{fig:008}, then
    \[
        {\rm ind}(a)+{\rm ind}(b)+{\rm ind}(c)=1.\label{lem:a+b+c}
    \]
\ends\label{lemma:index}
\end{lemma}
\subsection{S-graphs and flips}
The notation of S-graphs is the dual of mixed-angulations, cf. \cite{HKK,CHQ}.

\begin{definition}
Let $\CAS=\{\CA_1,...,\CA_n\}$ be a collection of graded admissible closed arcs on $\gmsx$. We call $\CAS$ an \emph{S-graph} if
\numbers
    \item $\CA_1,...,\CA_n$ form a full formal closed arc system on $\gmsx$;
    \item $\oIntd(\CA_i,\CA_j)=0$ for any $1\leq i,j\leq n$ and any $\rho\leq 0$.
\ends
\end{definition}

\begin{remark}
A collection $\CAS=\{\CA_1,...,\CA_n\}$ of graded closed arcs on $\gmsx$ is an S-graph if and only if there are representatives of $\CA_1,...,\CA_n$ in a minimal position pairwisely such that they form a full formal closed arc system on $\mathbf{S}^{\lambda}$ and all of the intersection indices of them at points on $\partial\mathbf{S}$ are at least 1.
\end{remark}

\begin{remark}
For the associated graded skew-gentle algebra $\Lambda_{\ac}$, we denote the canonical heart of $\mathcal{D}^b(\Lambda_{\ac})$ (i.e. the module category ${\rm mod}(\Lambda_{\ac})$) by $\mathcal{H}_{\ac}$. Then by the intersection formula, $X^{-1}(\mathcal{H}_{\ac})$ is an S-graph, which is called the initial S-graph, where $X$ is the bijection in \Cref{thm:int formula}. \label{initial_heart}
\end{remark}


Now we have the existence of an S-graph. Then we proceed to introduce the flips of S-graphs on $\gmsx$.

In this subsection, we fix representatives of arcs in an S-graph $\CAS$ in a minimal position pairwisely such that all of the intersection indices between them are at least 1.

\begin{lemma}
Let $\mathbb{S}$ be an S-graph on $\gmsx$ and $\wa,\,\CA\in\mathbb{S}$. Then we have $\oInt^1(\wa,\CA)\leq 2$.\label{Ext1}
\end{lemma}
\begin{proof}
Since $\wa$ and $\CA$ can only intersect at $\partial\mathbf{S}$, we have $|\overrightarrow{\cap}^1(\wa,\CA)|\leq 4$.

If $|\overrightarrow{\cap}^1(\wa,\CA)|>2$, then all their endpoints must coincide, denoted by $M\in\partial\mathbf{S}$. In this case, $\wa$ and $\CA$ are not twins. By Lemma \ref{lem:twin}, the intersections in $\overrightarrow{\cap}(\wa,\CA)$ are genuine-intersections. By the definition of an S-graph and $1^{\circ}$ of Lemma \ref{lemma:index}, there are at most three angles at $M$ with intersection index one, and at least one of them is from $\CA$ to $\wa$, cf. \Cref{fig:segments}, which is a contradiction.

Therefore, we have $|\overrightarrow{\cap}^1(\wa,\CA)|\leq 2$.
\end{proof}

\begin{figure}[htbp]\centering
	\def\dexc{black!10!blue!50!green}
	\begin{tikzpicture}[scale=0.9]
	\draw[ultra thick]plot [smooth,tension=1] coordinates {(-2,0.2) (0,0) (2,0.2)};
	\draw[red] (0,0) .. controls +(-120:1) and +(30:1) .. (-3,-2.5);
	\draw[red] (0,0) .. controls +(-80:0.5) and +(60:0.5) .. (-1.2,-3);
	\draw[red] (0,0) .. controls +(-60:0.5) and +(70:1) .. (1,-3);
	\draw[red] (0,0) .. controls +(-50:0.5) and +(120:1) .. (3,-2);
	\draw[Green,->, thick,>=stealth]plot
	(-62:0.65)to[bend left=25](-100:0.65);
    \draw[Green,->, thick,>=stealth]plot
    (-102:0.7)to[bend left=15](-126:0.7);
    \draw[Green,->, thick,>=stealth]plot (-35:0.7)to[bend left=15](-61:0.7);
    \draw(0,0)\DEC;
    \draw (0,0.3)node{$M$};
    \draw (0,-1)node{1};
    \draw (-0.5,-0.85)node{1};
    \draw (0.7,-0.8)node{1};
    \end{tikzpicture}
    \caption{Segments of $\wa$ and $\CA$}\label{fig:segments}
\end{figure}

For any $\wa, \CA\in\CAS$, if $\wa\neq\CA$ and $\oInt^1(\wa,\CA)>0$, we know that $\oInt^1(\wa,\CA)$ is 1 or 2 by Lemma \ref{Ext1}. Let $\overrightarrow{\cap}^1(\wa,\CA)$ be $\{a\}$ or $\{a_1,a_2\}$ respectively. Then we define $\wa\wedge^1\CA:=\wa\wedge_a\CA$ when $\oInt^1(\wa,\CA)=1$ and $\wa\wedge^1\CA:=(\wa\wedge_{a_1}\CA)\wedge_{a_2}\CA$ when $\oInt^1(\wa,\CA)=2$. By convenience, we define $\wa\wedge^1\CA:=\wa$ if $\oInt^1(\wa,\CA)=0$.

\begin{remark}
When $\oInt^1(\wa,\CA)=2$, let $a_1$ be the angle at $\wa(0)=\CA(0)$ and $a_2$ be the angle at $\wa(1)=\CA(1)$ (the two angles may come from the same endpoint or different endpoints). Then $(\wa\wedge_{a_1}\CA)\wedge_{a_2}\CA$ and $(\wa\wedge_{a_2}\CA)\wedge_{a_1}\CA$ are both homotopic to $(\CA^{-1}\cdot\wa)\cdot\CA^{-1}$, where $\cdot$ is the product of paths. Moreover, the gradings of $(\wa\wedge_{a_1}\CA)\wedge_{a_2}\CA$ and $(\wa\wedge_{a_2}\CA)\wedge_{a_1}\CA$ are both inherited from $\wa$ by Definition \ref{smoothing}. Therefore,
\[(\wa\wedge_{a_1}\CA)\wedge_{a_2}\CA = (\wa\wedge_{a_2}\CA)\wedge_{a_1}\CA,\]
and $\wa\wedge^1\CA$ is well-defined.\label{homotopy}
\end{remark}

\begin{definition}
Let $\CAS$ be an S-graph and $\CA\in\CAS$. We define the \emph{forward flip} of $\CAS$ with respect to $\CA$ as $\CAS^{\sharp}_{\CA}:=\{\wa^{\sharp}_{\CA}|\,\wa\in\CAS\}$ (cf. \Cref{Examples}), where
\[\wa_{\CA}^{\sharp}:=\begin{cases}
\CA[1], & \text{if }\wa=\CA,\\
\wa\wedge^1\CA, & \text{otherwise.}
\end{cases}\]
\label{flip}
\end{definition}

\begin{remark}
Similarly, we can define $\wa\wedge^{-1}\CA$ for any $\wa, \CA\in\CAS$ with $\wa\neq\CA$. Moreover, we define the \emph{backward flip} of $\CAS$ with respect to $\CA$ as $\CAS^{\flat}_{\CA}:=\{\wa^{\flat}_{\CA}|\,\wa\in\CAS\}$, where
\[\wa_{\CA}^{\flat}:=\begin{cases}
\CA[-1], & \text{if }\wa=\CA,\\
\wa\wedge^{-1}\CA, & \text{otherwise.}
\end{cases}\]
This is in fact the reverse operation of the forward flip, i.e. $(\mathbb{S}_{\CA}^{\sharp})_{\CA[1]}^{\flat}=\mathbb{S}$, but we don't need this.
\end{remark}

\begin{remark}
For the case $\wa\neq\CA$ and $\oInt^1(\wa,\CA)=1$, since $\wa$ and $\CA$ do not intersect in $\surf^{\circ}$, $\wa^{\sharp}_{\CA}$ has no self-intersection in $\surf^{\circ}$. Then $\wa$, $\CA$ and $\wa^{\sharp}_{\CA}$ form a contractible triangle on $\gmsx$ with interior angles $a\in\overrightarrow{\cap}(\wa,\CA), b\in\overrightarrow{\cap}(\CA,\wa^{\sharp}_{\CA})$ and $c\in\overrightarrow{\cap}(\wa^{\sharp}_{\CA},\wa)$. We have ${\rm ind}(a)=1$ by our assumption and ${\rm ind}(c)=0$ by Definition \ref{smoothing}. Then we have ${\rm ind}(b)=0$ by $2^{\circ}$ of Lemma \ref{lemma:index}. Therefore, the gradings of the arc segments of $\wa^{\sharp}_{\CA}$ are inherited from $\wa$ and $\CA$.

Similarly for other cases, we can get a general conclusion: the gradings of the arc segments of $\wa^{\sharp}_{\CA}$ are inherited from $\wa$, $\CA$ and $\CA[1]$ (only in the case $\wa=\CA$). \label{inherite}
\end{remark}

\begin{figure}[htbp]
	\begin{tikzpicture}[scale=.35]
	\draw[ultra thick]plot [smooth,tension=1] coordinates {(3,-4.5) (-90:5.2) (-3,-4.5)};
	\draw[ultra thick]plot [smooth,tension=1] coordinates {(-6,-9) (-5,-10.5) (-4,-13)};
	\draw[ultra thick]plot [smooth,tension=1] coordinates {(6,-9) (5,-10.5) (4,-13)};
	\draw[red] (-90:5.2) .. controls +(-120:3) and +(30:3) .. (-5,-10.5);
	\draw[red] (-90:5.2) .. controls +(-60:3) and +(150:3) .. (5,-10.5);
	\draw[cyan, ultra thick] (-5,-10.5) .. controls +(20:3) and +(160:3) .. (5,-10.5);
    \draw[Green,->, thick,>=stealth] (-86:6) to[bend left = 30] (-95:6);
    \draw (-90:5.2)\DEC;
	\draw (-5,-10.5)\DEC;
	\draw (5,-10.5)\DEC;
    \draw (-90:6.75) node {1} ;
    \draw (-90:10.8) node {$\wa_{\CA}^{\sharp}$} ;
    \draw (4,-9) node {$\wa$} ;
    \draw (-4,-9) node {$\CA$} ;

	\begin{scope}[shift={(13,0)}]
	\draw[ultra thick]plot [smooth,tension=1] coordinates {(3,-4.5) (-90:5.2) (-3,-4.5)};
	\draw[ultra thick,fill=gray!10] (-90:11) ellipse (1.5);
	\draw[red] (-90:5.2) .. controls +(-45:6) and +(-30:5) .. (-90:12.5);
	\draw[red] (-90:5.2) .. controls +(-135:6) and +(-150:5) .. (-90:12.5);
    \draw[cyan, ultra thick] (-90:5.2)
        .. controls +(-145:8) and +(160:5) .. (0,-15)
        .. controls +(-20:4) and +(0:4) .. (0,-8)
        .. controls +(-170:3) and +(-170:3) .. (0,-12.5);
    \draw[Green,->, thick,>=stealth] (-84:6) to[bend left = 30] (-96:6);
    \draw[Green,->, thick,>=stealth] (-86:13) to[bend left = 50] (-94:13);
    \draw (-90:12.5)\DEC;
	\draw (-90:5.2)\DEC;
    \draw[blue] (-90:9.5)\nn;
    \draw (-90:14) node {1} ;
    \draw (-90:6.75) node {1} ;
    \draw (-80:16) node {$\wa_{\CA}^{\sharp}$} ;
    \draw (4,-11) node {$\wa$} ;
    \draw (-3.6,-11) node {$\CA$} ;
    \end{scope}
	\end{tikzpicture}
	\caption{Examples for the forward flip}\label{Examples}
    \end{figure}

We aim to show that the forward flip of an S-graph is also an S-graph.

\begin{lemma}
For an S-graph $\CAS$ and $\CA\in\CAS$, the forward flip $\CAS_{\CA}^{\sharp}$ is a full formal closed arc system.\label{lem:ffas}
\end{lemma}

\begin{proof}
By flipping closed arcs $\CA_1,...,\CA_n$ in $\mathbb{S}$ one by one, we get a sequence $\mathbb{S}=\mathbb{S}_0\to\mathbb{S}_1\to\cdots\to\mathbb{S}_n=\mathbb{S}_{\CA}^{\sharp}$, where $\mathbb{S}_{i+1}=(\mathbb{S}_i\backslash\{\CA_i\})\cap\{(\CA_i)_{\CA}^{\sharp}\}$ for $1\leq i\leq n$. We prove our statement by induction. Assuming that $\mathbb{S}_j$ is a full formal closed arc system for some $1\leq j\leq n$, and then we consider $\mathbb{S}_{j+1}=(\mathbb{S}_j\backslash\{\wa\})\cap\{\wa_{\CA}^{\sharp}\}$ (we denote $\CA_j$ by $\wa$ here).

For case $\wa=\CA$ and case $\oInt^1(\wa,\CA)=0$, the underlying ungraded arcs in $\mathbb{S}_j$ all remain and hence $\mathbb{S}_{j+1}$ is still a full formal arc system. Now assume that $\wa\neq\CA$ and $\oInt^1(\wa,\CA)\neq 0$.

We claim that any $a\in\overrightarrow{\cap}^1(\wa,\CA)$ can not be cut out by any $\wb\in\mathbb{S}_j$ unless $\wb$ is a twin of $\CA$ or $\wa$. If $\wb\in\mathbb{S}_j$ is not a twin of $\CA$ or $\wa$, let $a=b+c$ for some $b\in\overrightarrow{\cap}(\wa,\wb)$ and $c\in\overrightarrow{\cap}(\wb,\CA)$. Since the gradings of the arc segments of $\wb$ are inherited from the arcs in $\mathbb{S}$ by Remark \ref{inherite}, we have ${\rm ind}(b)\geq 1$ and ${\rm ind}(c)\geq 1$. Then ${\rm ind}(a)\geq 2$ by $1^{\circ}$ of Lemma \ref{lemma:index}, which contradicts the fact that $a\in\overrightarrow{\cap}^1(\wa,\CA)$.

\textbf{Case 1:} any $a\in\overrightarrow{\cap}^1(\wa,\CA)$ is not cut out by arcs in $\mathbb{S}_j$.

First we prove that arcs in $\CAS_{j+1}$ can not intersect each other in $\mathbf{S}^{\circ}$. Since $\wa$ and $\CA$ do not intersect in $\surf^{\circ}$, $\wa^{\sharp}_{\CA}$ has no self-intersection in $\surf^{\circ}$. Then for any $\wb\in\CAS_j$, we need to show that it can not intersect $\wa_{\CA}^{\sharp}$ in $\mathbf{S}^{\circ}$. By the definition of $\wa_{\CA}^{\sharp}$, it is the smoothing out of $\wa$ and $\CA$ at angles of index 1. Then if $\wb$ intersects $\wa_{\CA}^{\sharp}$ in $\mathbf{S}^{\circ}$, we have $\wb$ intersects $\wa$ or $\CA$ in $\mathbf{S}^{\circ}$ or some angle $a\in\overrightarrow{\cap}^1(\wa,\CA)$ is cut out by $\wb$. But by the inductive assumption, $\wb$ can not intersect $\wa$ or $\CA$ in $\mathbf{S}^{\circ}$. And in our case any $a\in\overrightarrow{\cap}^1(\wa,\CA)$ is not cut out by arcs in $\mathbb{S}_j$, which is a contradiction.

Next We show that after flipping $\wa$ to $\wa_{\CA}^{\sharp}$, $\mathbb{S}_{j+1}$ still divides $\gmsx$ into polygons and each polygon contains exactly one open marked point. By the induction hypothesis, we know that $\mathbb{S}_j$ is a full formal closed arc system. This means that $\gmsx$ is divided into $\mathbb{S}_j$-polygons. When we flip $\wa$ to $\wa_{\CA}^{\sharp}$, the $\mathbb{S}_j$-polygons that do not contain $\wa$ as an edge will remain unchanged. However, there are exactly two $\mathbb{S}_j$-polygons, let's call them $\mathbb{D}_1$ and $\mathbb{D}_2$, that do contain $\wa$ as an edge. Then for any angle $a\in\overrightarrow{\cap}^1(\wa,\CA)$, since $a$ is not cut out by arcs in $\mathbb{S}_j$, $a$ is an interior angle of $\mathbb{D}_1$ or $\mathbb{D}_2$.

We have $\oInt^1(\wa,\CA)\leq 2$ by Lemma \ref{Ext1}. Then there are two subcases:
\begin{itemize}
    \item $\oInt^1(\wa,\CA)=1$: let $\overrightarrow{\cap}^1(\wa,\CA)=\{a\}$ and $a$ is an interior angle of $\mathbb{D}_1$. When we flip $\wa$ to $\wa_{\CA}^{\sharp}$, the polygon $\mathbb{D}_1$ leaves two edges $\wa$ and $\CA$ and gets an edge $\wa_{\CA}^{\sharp}$, while $\mathbb{D}_2$ leaves $\wa$ and gets $\CA$ and $\wa_{\CA}^{\sharp}$.
    \item $\oInt^1(\wa,\CA)=2$: let $\overrightarrow{\cap}^1(\wa,\CA)=\{a_1,\,a_2\}$, where $a_1$ is an interior angle of $\mathbb{D}_1$ and $a_2$ is an interior angle of $\mathbb{D}_2$. When we flip $\wa$ to $\wa_{\CA}^{\sharp}$, $\mathbb{D}_1$ and $\mathbb{D}_2$ both leave two edges $\CA$ and $\wa$ in clockwise and get two edges $\wa_{\CA}^{\sharp}$ and $\CA$ in clockwise.
\end{itemize}
In both subcases, $\mathbb{D}_1$ and $\mathbb{D}_2$ are still polygons after flipping $\wa$ to $\wa_{\CA}^{\sharp}$. Moreover, since the boundary edges in $\mathbb{D}_1$ and $\mathbb{D}_2$ do not change, each of $\mathbb{D}_1$ and $\mathbb{D}_2$ still contains exactly one open marked point. Thus $\mathbb{S}_{j+1}$ still divides $\gmsx$ into polygons and each polygon contains exactly one open marked point.

Now we see that $\mathbb{S}_{j+1}$ is a full formal closed arc system when each $a\in\overrightarrow{\cap}^1(\wa,\CA)$ is not cut out by arcs in $\mathbb{S}_j$.

\textbf{Case 2:} some $a\in\overrightarrow{\cap}^1(\wa,\CA)$ is cut out by $\wb\in\mathbb{S}_j$.

We have shown that $\wb$ is a twin of $\CA$ or $\wa$. There exist two subcases:
\begin{itemize}
    \item $\wb$ is the twin $\CA'$ of $\CA$, cf. pictures (a) in \Cref{fig:003};
    \item $\wb$ is the twin $\wa'$ of $\wa$, cf. pictures (b) in \Cref{fig:003}.
\end{itemize}
Without loss of generality, we assume that $\wb$ is the twin $\CA'$ of $\CA$.

\begin{figure}[htbp]
	\begin{tikzpicture}[scale=.3]
    \draw[ultra thick]plot [smooth,tension=1] coordinates {(3,-4.5) (-90:5.2) (-3,-4.5)};
	\draw[ultra thick,fill=gray!10] (-90:11) ellipse (1.5);
	\draw (-90:11) node {$\Vot$};
	\draw[red] (-90:5.2) .. controls +(-45:6) and +(-30:5) .. (-90:12.5);
	\draw[red] (-90:5.2) .. controls +(-135:6) and +(-150:5) .. (-90:12.5);
	\draw[red] (-90:5.2) .. controls +(-15:3) and +(130:3) .. (-60:11);
	\draw (-90:12.5)\DEC;
	\draw (-90:5.2)\DEC;
	\draw (-60:11)\DEC;
    \draw[blue] (-90:9.5)\nn;
    \draw[Green,->, thick,>=stealth]plot [smooth,tension=1] coordinates
    {(-75:6) (-87:6.3) (-97:6)};
    \draw (0,-7) node {$a$} ;
    \draw (4,-11) node {$\CA'$} ;
    \draw (-3.8,-11) node {$\CA$} ;
    \draw (5,-8) node {$\wa$} ;
    \draw (0,-15.5) node {(a)} ;

            \begin{scope}[shift={(16,0)}]
        	\draw[ultra thick]plot [smooth,tension=1] coordinates {(3,-4.5) (-90:5.2) (-3,-4.5)};
	\draw[ultra thick,fill=gray!10] (-90:11) ellipse (1.5);
	\draw (-90:11) node {$\Vot$};
	\draw[red] (-90:5.2) .. controls +(-45:6) and +(-30:5) .. (-90:12.5);
	\draw[red] (-90:5.2) .. controls +(-135:6) and +(-150:5) .. (-90:12.5);
	\draw[red] (-90:5.2) .. controls +(-165:3) and +(50:3) .. (-120:11);
	\draw (-90:12.5)\DEC;
	\draw (-90:5.2)\DEC;
	\draw (-120:11)\DEC;
    \draw[blue] (-90:9.5)\nn;
    \draw[Green,->, thick,>=stealth]plot [smooth,tension=1] coordinates
    {(-84:6) (-93:6.32) (-105:6)};
    \draw (0,-7) node {$a$} ;
    \draw (4,-11) node {$\wa$} ;
    \draw (-3.8,-11) node {$\wa'$} ;
    \draw (-5,-8) node {$\CA$} ;
    \draw (0,-15.5) node {(b)} ;
    \end{scope}
	\end{tikzpicture}
	\caption{Binary subcases}\label{fig:003}
    \end{figure}

Let $P$ be the common endpoint of $\CA$ and $\CA'$ on the binary $\Vot$ and $M$ be the another common endpoint. First we show that the angle from $\CA'$ to $\CA$ at $P$ can not be cut out by arcs in $\CAS_j$. Let angles $b\in\overrightarrow{\cap}(\wa,\CA')$ and $c\in\overrightarrow{\cap}(\CA',\CA)$ at $M$ satisfy $a=b+c$. Since ${\rm ind}(a)=1$ and ${\rm ind}(b)\geq 1$, by $1^{\circ}$ of Lemma \ref{lemma:index}, we have ${\rm ind}(c)\leq 0$. Thus the angle from $\CA'$ to $\CA$ at $P$ also has non-positive index by Remark \ref{winding_binary}, and then can not be divided into two angles with index both at least 1. Thus the angle at $P$ can not be cut out by arcs in $\CAS_j$. And as a corollary, we have $\oInt^1(\wa,\CA)=1$.

Next we show that after changing representatives, there are no intersections in $\surf^{\circ}$ between arcs in $\CAS_{j+1}$. Replacing $\CA'$ by ${\rm D}_{\Vot}^{-2}\cdot\CA'$, we see that there is no intersection in $\mathbf{S}^{\circ}$ between $\wa_{\CA}^{\sharp}$ and ${\rm D}_{\Vot}^{-2}\cdot\CA'$, cf. \Cref{fig:005}. Moreover, we claim that there is no intersection in $\mathbf{S}^{\circ}$ between ${\rm D}_{\Vot}^{-2}\cdot\CA'$ and any other arc in $\CAS_j$. Since $\mathbb{S}_j$ is a full formal closed arc system, if ${\rm D}_{\Vot}^{-2}\cdot\CA'$ and some other $\wg\in\mathbb{S}_j$ intersects in $\mathbf{S}^{\circ}$, then $\wg$ cuts out the angle from $\CA'$ to $\CA$ at $P$, which is a contradiction. Therefore, we can choose ${\rm D}_{\Vot}^{-2}\cdot\CA'$ as the representative of $\CA'$ such that arcs in $\mathbb{S}_{j+1}$ have no intersections in $\mathbf{S}^{\circ}$.

\begin{figure}[htbp]
	\begin{tikzpicture}[scale=.3]
    \draw[ultra thick]plot [smooth,tension=1] coordinates {(3,-4.5) (-90:5.2) (-3,-4.5)};
    \draw[ultra thick]plot [smooth,tension=1] coordinates {(6.2,-6.5) (-60:11) (6.2,-12.5)};
	\draw[ultra thick,fill=gray!10] (-90:11) ellipse (1.5);
	\draw (-90:11) node {$\Vot$};
	\draw[red] (-90:5.2) .. controls +(-45:6) and +(-30:5) .. (-90:12.5);
	\draw[red] (-90:5.2) .. controls +(-135:6) and +(-150:5) .. (-90:12.5);
	\draw[red] (-90:5.2) .. controls +(-15:3) and +(130:3) .. (-60:11);
	\draw (-90:12.5)\DEC;
	\draw (-90:5.2)\DEC;
	\draw (-60:11)\DEC;
    \draw[blue] (-90:9.5)\nn;
    \draw[Green,->, thick,>=stealth]plot [smooth,tension=1] coordinates
    {(-75:6) (-87:6.3) (-97:6)};
    \draw (0,-7) node {$a$} ;
    \draw (4,-11) node {$\CA'$} ;
    \draw (-3.8,-11) node {$\CA$} ;
    \draw (3.7,-6.5) node {$\wa$} ;
    \draw (9,-9) node {$\Longrightarrow$};
    \begin{scope}[shift={(16,0)}]
    \draw[ultra thick]plot [smooth,tension=1] coordinates {(3,-4.5) (-90:5.2) (-3,-4.5)};
    \draw[ultra thick]plot [smooth,tension=1] coordinates {(6.2,-6.5) (-60:11) (6.2,-12.5)};
	\draw[ultra thick,fill=gray!10] (-90:11) ellipse (1.5);
	\draw (-90:11) node {$\Vot$};
	\draw[red] (-90:5.2) .. controls +(-45:6) and +(-30:5) .. (-90:12.5);
	\draw[red] (-90:5.2) .. controls +(-135:6) and +(-150:5) .. (-90:12.5);
	\draw[cyan, ultra thick] (-60:11)
        .. controls +(160:1) and +(30:3) ..  (0,-7.5)
        .. controls +(-150:5) and +(-155:4) .. (0,-12.5);
	\draw (-90:12.5)\DEC;
	\draw (-90:5.2)\DEC;
	\draw (-60:11)\DEC;
    \draw[blue] (-90:9.5)\nn;
    \draw (4,-11) node {$\CA'$} ;
    \draw (-3.8,-11) node {$\CA$} ;
    \draw (3.7,-6.9) node {$\wa_{\CA}^{\sharp}$} ;
    \draw (9,-9) node {$\Longrightarrow$};
    \end{scope}
    \begin{scope}[shift={(32,0)}]
    \draw[ultra thick]plot [smooth,tension=1] coordinates {(3,-4.5) (-90:5.2) (-3,-4.5)};
    \draw[ultra thick]plot [smooth,tension=1] coordinates {(6.2,-6.5) (-60:11) (6.2,-12.5)};
	\draw[ultra thick,fill=gray!10] (-90:11) ellipse (1.5);
	\draw (-90:11) node {$\Vot$};
	\draw[red] (-90:5.2) .. controls +(-135:6) and +(-150:5) .. (-90:12.5);
	\draw[cyan, ultra thick] (-60:11)
        .. controls +(160:1) and +(30:3) ..  (0,-7.5)
        .. controls +(-150:5) and +(-155:4) .. (0,-12.5);
	\draw[red] (-90:5.2)
        .. controls +(-150:7) and +(-150:9) .. (-90:14.2)
        .. controls +(30:4) and +(20:2.5) .. (0,-8.5)
        .. controls +(-160:4) and +(-165:3) .. (0,-12.5);
	\draw (-90:12.5)\DEC;
	\draw (-90:5.2)\DEC;
	\draw (-60:11)\DEC;
    \draw[blue] (-90:9.5)\nn;
    \draw (-3.8,-11) node {$\CA$} ;
    \draw (3.7,-6.9) node {$\wa_{\CA}^{\sharp}$} ;
    \draw (2,-14.9) node{${\rm D}_{\Vot}^{-2}\cdot\CA'$};
    \end{scope}
	\end{tikzpicture}\vskip -.5cm
	\caption{Forward flip and ${\rm D}_{\Vot}^{-2}$-action}\label{fig:005}
    \end{figure}

Moreover, after choosing ${\rm D}_{\Vot}^{-2}\cdot\CA'$ as the representative of $\CA'$, we only consider the two $\CAS_j$-polygons containing $\wa$ as an edge as before. We get that the arc system $\mathbb{S}_{j+1}$ still divides $\gmsx$ into polygons and each polygon contains exactly one open marked point.

Now we see that $\mathbb{S}_{j+1}$ is a full formal closed arc system in the case that some $a\in\overrightarrow{\cap}^1(\wa,\CA)$ is not cut out by $\wb\in\mathbb{S}_j$.

By induction, we know that $\mathbb{S}_{\CA}^{\sharp}$ is a full formal closed arc system. In all, the proof is completed.
\end{proof}

Next we investigate the intersection indices between arcs in $\mathbb{S}_{\CA}^{\sharp}$.

\begin{lemma}
For an S-graph $\CAS$ and an arc $\CA\in\CAS$, we have $\oInt^{\rho}(\CA,\CA)=0$ for any $\rho\leq 1$.\label{lemma:oInt^1=0}
\end{lemma}

\begin{proof}
By the definition of an S-graph, we have $\oInt^{\rho}(\CA,\CA)=0$ for any $\rho\leq 0$. Thus we only need to show $\oInt^{1}(\CA,\CA)=0$. If $\oInt^{1}(\CA,\CA)\neq 0$, since $\CA$ can not intersect itself in $\surf^{\circ}$, let $\overrightarrow{\cap}^1(\CA,\CA)=\{a\}$, where the vertex of the angle $a$ is on $\partial\surf$. Then $a$ can not be cut out by other arcs in $\CAS$ by $1^{\circ}$ of Lemma \ref{lemma:index}, and thus $a$ is an interior angle of an $\mathbb{S}$-polygon $\mathbb{D}$. But now $\mathbb{D}$ has only one edge $\CA$ and thus does not contain a boundary edge, which is a contradiction.
\end{proof}

\begin{proposition}
Any forward flip $\CAS_{\CA}^{\sharp}$ of an S-graph $\mathbb{S}$ with respect to $\CA\in\mathbb{S}$ is still an S-graph.\label{StoS}
\end{proposition}

\begin{proof}
By Lemma \ref{lem:ffas}, $\mathbb{S}_{\CA}^{\sharp}$ is a full formal closed arc system. We only need to show that $\oInt^{\rho}(\wa_{\CA}^{\sharp},\wb_{\CA}^{\sharp})=0$ for any $\rho\leq 0$ and any $\wa,\wb\in\mathbb{S}$.

If $\wa_{\CA}^{\sharp}$ and $\wb_{\CA}^{\sharp}$ are twins, $\oInt^{\rho}(\wa_{\CA}^{\sharp},\wb_{\CA}^{\sharp})=0$ for any $\rho\leq 0$ by Lemma \ref{lem:twin}. Now we assume $\wa_{\CA}^{\sharp}$ and $\wb_{\CA}^{\sharp}$ are not twins.

\begin{itemize}
    \item[\textbf{Case 1:}] $\wb\neq\CA$. Then the gradings of arc segments of $\wa_{\CA}^{\sharp}$ are inherited from $\wa$, $\CA$ or $\CA[1]$, and the gradings of arc segments of $\wb_{\CA}^{\sharp}$ are inherited from $\wb$ or $\CA$ by Remark \ref{inherite}. Hence, intersection indices from $\wa_{\CA}^{\sharp}$ to $\wb_{\CA}^{\sharp}$ are at least 1 by the definition of an S-graph.
    \item[\textbf{Case 2:}] $\wb=\CA$.
    \begin{itemize}
        \item[\textbf{Case 2.1:}] $\oInt^1(\wa,\CA)=0$. In this subcase, $\wa_{\CA}^{\sharp}=\wa$, $\wb_{\CA}^{\sharp}=\CA[1]$, and intersection indices from $\wa$ to $\CA$ are at least 2. Then intersection indices from $\wa$ to $\CA[1]$ are at least 1.
        \item[\textbf{Case 2.2:}] $\oInt^1(\wa,\CA)>0$. We know that the gradings of arc segments of $\wa_{\CA}^{\sharp}$ are inherited from $\wa$, $\CA$ or $\CA[1]$. Since $\oInt^{\rho}(\CA,\CA)=0$ for any $\rho\leq 1$ by Lemma \ref{lemma:oInt^1=0}, we have $\oInt^{\rho}(\CA,\CA[1])=0$ and $\oInt^{\rho}(\CA[1],\CA[1])=0$ for any $\rho\leq 0$. And we have $\oInt^{\rho}(\wa,\CA[1])=0$ for any $\rho\leq -1$ by the definition of an S-graph. Moreover, if the grading of an arc segment of $\wa_{\CA}^{\sharp}$ are inherited from $\wa$, by the definition of the forward flip, there is not an intersection of index 1 from this arc segment to $\CA$, otherwise we should take smoothing out. Therefore, the intersection indices from any arc segment of $\wa_{\CA}^{\sharp}$ to $\CA[1]$ are at least 1.
    \end{itemize}
\end{itemize}

In conclusion, we have $\oInt^{\rho}(\wa_{\CA}^{\sharp},\wb_{\CA}^{\sharp})=0$ for any $\rho\leq 0$ as claimed.
\end{proof}

\subsection{Isomorphism between exchange graphs}

\begin{definition}[cf. {\cite{HKK,CHQ}}]
\label{def:exchange graph}
We define the \emph{exchange graph} ${\rm EG}(\gmsx)$ \emph{of S-graphs} on $\gmsx$ to be an oriented graph whose vertices are S-graphs on $\gmsx$ and edges are the forward flips defined as Definition \ref{flip}.

Moreover, we define the \emph{principal component} ${\rm EG}^{\circ}(\gmsx)$ of ${\rm EG}(\gmsx)$ to be the connected component of ${\rm EG}(\gmsx)$ containing the initial S-graph $X^{-1}(\mathcal{H}_{\ac})$.
\end{definition}

By Proposition \ref{StoS}, ${\rm EG}(\gmsx)$ is well-defined and $(n,n)$-regular.

Now we proceed to give our first main result: the isomorphism between the oriented graphs ${\rm EG}^{\circ}(\gmsx)$ and ${\rm EG}^{\circ}\mathcal{D}^b(\sg)$. This theorem gives the geometric realization of hearts and their simple tilting in $\mathcal{D}^b(\sg)$.

\begin{theorem}
The bijection $X$ in \Cref{thm:int formula} induces an isomorphism of oriented graphs
\[\begin{array}{rcl}
X:\,\,\,\,{\rm EG}^{\circ}(\gmsx) & \to & {\rm EG}^{\circ}\mathcal{D}^b(\sg) \\
\mathbb{S}=\{\CA\} & \mapsto & \mathcal{H}_{\mathbb{S}}=\langle X_{\CA}|\,\CA\in\CAS\rangle.
   \end{array}
\]
\label{main}
\end{theorem}

The proof of \Cref{main} will be given in the next subsection.

\subsection{Proof of \texorpdfstring{\Cref{main}}{}}
\label{classify}
We use induction. For the starting case, we know that $X$ in \Cref{thm:int formula} gives the correspondence between the initial S-graph $X^{-1}(\mathcal{H}_{\ac})$ of $\gmsx$ and the canonical heart $\mathcal H_{\ac}$ of $\mathcal{D}^b(\sg)$ by Remark \ref{initial_heart}.

Now we assume that $X(\mathbb{S})={\rm Sim}\,\mathcal{H}$ for an S-graph $\mathbb{S}$ on $\gmsx$ and a finite heart $\mathcal{H}$ in $\mathcal{D}^b(\sg)$. We only need to show that $X(\mathbb{S}_{\CA}^{\sharp})={\rm Sim}\,\mathcal{H}_{X_{\CA}}^{\sharp}$ for any $\CA\in\mathbb{S}$.

For any $S,T\in{\rm Sim}\,\mathcal{H}$, consider the simple tilting:
\[\begin{array}{rcl}
    \h & \xrightarrow{S} & \h_S^\sharp \\
     \{\XX\} & \mapsto & \{\XX^\sharp\},
   \end{array}
\]
where the simple $\XX$ in $\h$ becomes the simple $\XX^\sharp=\psi_S^{\sharp}(\XX)$ in $\h_S^\sharp$ by the simple tilting formula (cf. \cite[Proposition 5.4]{KQ}). We claim that $\XX^{\sharp}$ is an arc object and $\CA_{\XX^{\sharp}}=(\CA_T)_{\CA_S}^{\sharp}$.

By the simple tilting formula, we have a distinguished triangle
\[T\to S[1]\otimes\mathrm{Ext}^1(T,S)\to T^{\sharp}\to T[1].\]

If $\XX=S$, we have $\XX^{\sharp}=S[1]$. Thus $\CA_{T^{\sharp}}=\CA_S[1]=(\CA_S)_{\CA_S}^{\sharp}$.

If $\XX\neq S$ and $\mathrm{Ext}^1(\XX,S)=0$, we have $\XX^{\sharp}=\XX$. Thus $\CA_{T^{\sharp}}=\CA_T=(\CA_T)_{\CA_S}^{\sharp}$.

What is left to consider is the case that $\XX\neq S$ and ${\rm dim}\,\mathrm{Ext}^1(\XX,S)>0$. By Lemma \ref{Ext1}, we have ${\rm dim}\,\mathrm{Ext}^1(\XX,S)=1$ or $2$. Let us consider the Grothendieck group. 

From the distinguished triangle above, we have
\[[T^{\sharp}]=-[T]-\mathrm{dim}\,\mathrm{Ext}^1(T,S)[S]\]
in the Grothendieck group $K_0(\mathcal{D}^b(\sg))$. By Lemma \ref{K0 vs H1}, we have
\[\xi^{-1}[T^{\sharp}]=-\xi^{-1}[T]-\mathrm{dim}\,\mathrm{Ext}^1(T,S)\xi^{-1}[S]\]
in $H_1(\gmsx,\Y,\tau)$.

By the induction hypothesis, $T$ and $S$ are arc objects. By \Cref{thm:int formula}, we have $\overrightarrow{\mathrm{Int}}^1(\CA_T,\CA_S)=\mathrm{dim}\,\mathrm{Ext}^1(T,S)$, and then
\[\xi^{-1}[T^{\sharp}]=-[\CA_{T}]-\oInt^1(\CA_T,\CA_S)[\CA_{S}].\]

When ${\rm dim}\,(\XX,S)^1=1$, the mutation $(\CA_T)_{\CA_S}^{\sharp}$ is homotopic to $\CA_S^{-1}\cdot\CA_T^{-1}$ relatively to the set $\Y$ of closed marked points. And when ${\rm dim}\,(\XX,S)^1=2$, the mutation $(\CA_T)_{\CA_S}^{\sharp}$ is homotopic to $(\CA_S^{-1}\cdot\CA_T^{-1})\cdot\CA_S^{-1}$ relatively to $\Y$ by Remark \ref{homotopy}. Then we have
\[[(\CA_T)_{\CA_S}^{\sharp}]=-[\CA_{T}]-\oInt^1(\CA_T,\CA_S)[\CA_{S}]\]
in $H_1(\gmsx,\Y,\tau)$, and thus
\[\xi^{-1}[T^{\sharp}]=[(\CA_T)_{\CA_S}^{\sharp}]\]
in $H_1(\gmsx,\Y,\tau)$.

That is, $(\CA_T)_{\CA_S}^{\sharp}$ is a representative of the homology class $\xi^{-1}[T^{\sharp}]$.

In order to finish the induction, we also need to show that $T^{\sharp}$ is an arc object. It will be done using case by case study in Appendix \ref{appendix-2}. In each cases there, we see that $(\CA_T)_{\CA_S}^{\sharp}$ is an unknotted arc, and then $T^{\sharp}$ is an arc object.

By induction, we have $X(\mathbb{S}_{\CA}^{\sharp})={\rm Sim}\,\mathcal{H}_{X(\CA)}^{\sharp}$.

Now we have finished the proof of \Cref{main}.
\section{Quadratic differentials as stability conditions}
\label{quad_vs_stab}
In this section, we will use the isomorphism $X$ in \Cref{main} to establish a correspondence between quadratic differentials and stability conditions by extending the isomorphism between two oriented graphs into an injection between these two moduli spaces, following the approach outlined in \cite{BS,BMQS,CHQ}. The necessary background on quadratic differentials is reviewed in \Cref{quadratic differentials}.
\subsection{Basic notions of quadratic differentials}
\label{quadratic differentials}
Let $C$ be a compact Riemann surface and $K_C$ be the canonical bundle over $C$. A \emph{holomorphic quadratic differential} (resp. \emph{meromorphic quadratic differential}) over $C$ is a holomorphic (resp. meromorphic) section of $K_C^{\otimes 2}$. Let $\varphi$ be a meromorphic quadratic differential on $C$ that is holomorphic and non-vanishing away from a finite subset $D\subset C$. Each point in $D$ is either a zero, a pole, a regular point, or a transcendental singularity of $\varphi$, of the form $\mathrm{exp}(f(z))z^{-l}g(z)dz^2$ where $f$ has a pole of order $k$ at $z=0$ and $g$ is a non-zero holomorphic function near $z=0$, called an \emph{exponential singularity of index $(k,l)$}. Thus, $D$ consists of the zeros, poles, and exponential singularities of $\varphi$, as well as possibly some additional marked points where $\varphi$ is holomorphic and non-vanishing.

In addition to the complex-analytic point of view on quadratic differentials, there is also a metric point of view. Specifically, $|\varphi|$ defines a flat Riemannian metric on the complement $C \backslash D$, which means that $C\backslash D$ has the structure of a metric space that is locally isometric to the Euclidean plane. However, $C\backslash D$ is typically not complete as a metric space, and its metric completion, denoted by $\overline{C\backslash D}$, involves adding a finite number of \emph{conical points}.
More precisely, $\overline{C\backslash D}$ has a conical point with cone angle $(n+2)\pi$ for each zero of order $n$ of $\varphi$, a conical point of cone angle $\pi$ for each simple pole of $\varphi$, and $n$ infinite-angle conical points for each exponential singularity of index $(n,m)$, for any $m\in\mathbb Z$, as well as a (non-singular) point with cone angle $2\pi$ for each regular point of $\varphi$ in $D$. Note that the metric is already complete near higher order poles of $\varphi$, so these do not lead to any additional conical points. The evidence supporting these assertions can be found in \cite{HKK}. The horizontal foliation of $\varphi$, denoted by $\mathrm{hor}(\varphi)$, is obtained by pulling back the horizontal foliation of normal charts on the Euclidean plane. The leaves of the horizontal foliation are referred to as horizontal trajectories of $\varphi$. We are particularly interested in the following types of trajectories:
\begin{itemize}
\item \emph{Generic trajectories}: avoid the conical points and converge ``to infinity'' (toward a higher order pole or exponential singularity) in both directions.
\item \emph{Saddle trajectories}: converge to conical points in both directions.
\item \emph{Separating trajectories}: approach a conical point in one direction and escape to infinity in the other direction.
\item \emph{Recurrent trajectories}: be recurrent in at least one direction.
\end{itemize}
In general, a \emph{saddle connection} is a maximal straight arc of some phrase $\theta$ which converge to conical points in both directions. A saddle trajectory is a horizontal saddle connection.

The surface $C$ can have a total area that is either finite or infinite, and the area of $\varphi$ is finite if and only if $\varphi$ has no higher order poles or exponential singularities. The study of quadratic differentials with finite area is of interest in ergodic theory. In contrast, quadratic differentials $\varphi$ with infinite area are closely related to full formal arc systems, as observed in \cite{HKK}. To elaborate, assuming that $|\varphi|$ has at least one conical point (excluding the flat plane and flat cylinder), a generic choice of $\theta\in\mathbb R$ yields a \emph{horizontal strip decomposition} of the horizontal foliation of $e^{i\theta}\varphi$. This implies that, after possibly replacing $\varphi$ by $e^{i\theta}\varphi$, we can assume that the horizontal foliation $\mathrm{hor}(\varphi)$ has no finite-length leaves, and thus contains only generic and separating trajectories.

By definition, each separating trajectory is associated with a conical point, and a conical point with cone angle $n\pi$ gives rise to $n$ separating trajectories. These trajectories divide the surface into \emph{horizontal strips} that are foliated by one-parameter families of generic trajectories. The \emph{height} of each horizontal strip is given by $a\in(0,\infty]$, and they are isometric to $(0,a)\times\mathbb R$. The boundary of a horizontal strip with finite (resp. infinite) height consists of four (resp. two) separating trajectories, some of which may be identified.

A \emph{flat surface} $(C,D,\varphi)$ is defined as a compact Riemann surface $C$ together with a quadratic differential $\varphi$ such that $D$ is a finite set of exponential singularities.

\subsection{Isomorphisms between moduli spaces}
We begin by reviewing the construction of graded marked surfaces from flat surfaces (cf. \cite{HKK, IQ2}).
Let $(C,D,\varphi)$ be a flat surface. The associated graded marked surface $\surf^{\lambda}(C,D,\varphi)=(\surf,\M,\Y,\lambda)$ is obtained by performing the real blow-up of $C$ at all the points in $D$, such that
\begin{enumerate}
\item each boundary component of $\surf$ correspond to the real blow-up of an exponential singularity of index $(k,l)$;
\item the set $\M$ of open marked points corresponds to the set of directions on $\partial\surf$ from which horizontal trajectories approach, where there are $k$ such directions with respect to the $\varphi$-metric completion for an exponential singularity of index $(k,l)$;
\item the set $\Y$ of closed marked points corresponds to the set of directions between the points in $\M$, and is therefore dual to the points in $M$ on the boundaries;
\item $\lambda$ is the horizontal foliation of $\varphi$.
\end{enumerate}
Furthermore, if we choose a subset of boundary components of $\surf$ that come from exponential singularities of index $(1,1)$ as the set of binaries $\vot$, then we obtain the GMSb $\gmsx(C,D,\varphi)=(\surf,\M,\Y,\vot,\lambda)$ associated to $(C,D,\varphi)$. It is worth noting that any GMSb $(\surf, \M, \Y, \vot, \lambda)$ can be realized as the real blow-up of some flat surface $(C, D, \varphi)$. Additionally, the choice of boundary components corresponding to $\vot$ is based on the winding number around a binary and \cite[Proposition 6.26]{IQ2}.

Assuming $\surf^{\lambda}$ is a GMS, an $\surf^{\lambda}$-framed quadratic differential (cf. \cite{HKK,IQ2}) on $\surf^{\lambda}$ consists of a flat surface $(C,D,\varphi)$ together with an isomorphism of marked surfaces $f:\surf^{\lambda}\to\surf^{\lambda}(C,D,\varphi)$. Two $\surf^{\lambda}$-framed quadratic differentials $(C,D,\varphi,f)$ and $(C',D',\varphi',f')$ are considered equivalent if there exists an isomorphism of marked surfaces $g:\surf^{\lambda}(C,D,\varphi)\to\surf^{\lambda}(C',D',\varphi)$, and $(f')^{-1}gf$ is isotopic to the identity as an isomorphism of graded marked surfaces. Denote by $\mathrm{FQuad}(\surf^{\lambda})$ the \emph{moduli space of $\surf^{\lambda}$-framed quadratic differentials}. We can now state the similar definition involving binaries.

\begin{definition}
\label{moduli-quad}
Let $\gmsx$ be a GMSb. We define the \emph{moduli space of $\gmsx$-framed quadratic differentials} $\mathrm{FQuad}(\gmsx)$ as follows. A point in $\mathrm{FQuad}(\gmsx)$ is represented by a flat surface $(C,D,\varphi)$ together with an isomorphism of marked surfaces with binary $f:\gmsx\to\gmsx(C,D,\varphi)$. Two points $(C,D,\varphi,f)$ and $(C',D',\varphi',f')$ are equivalent if
\begin{itemize}
\item there exists an isomorphism of marked surfaces with binary $g:\gmsx(C,D,\varphi)\to\gmsx(C',D',\varphi)$;
\item $(f')^{-1}gf$ is isotopic to an element in $\Dfang$.
\end{itemize}
\end{definition}

If we replace the second equivalent condition with
\begin{itemize}
\item $(f')^{-1}gf$ is isotopic to the identity,
\end{itemize}
then we obtain an open subset of $\mathrm{FQuad}(\surf^{\lambda})$ defined in \cite{HKK,IQ2}, denoted by $\mathrm{FQuad}_{\vot}(\surf^{\lambda})$, which satisfies the following relation:
\begin{equation}
\label{relation-moduli}
\mathrm{FQuad}(\gmsx) \cong \mathrm{FQuad}_{\vot}(\surf^{\lambda})/\Dfang,
\end{equation}
which means that $\mathrm{FQuad}(\gmsx)$ can be interpreted as a $\Dfang$-quotient of an open subet of $\mathrm{FQuad}(\surf^{\lambda})$.

Recall that there is a well-known affine structure on $\mathrm{FQaud}(\surf^{\lambda})$ given by the period map (cf. \cite{HKK,IQ2}). Using a similar method, an affine structure on $\mathrm{FQuad}(\gmsx)$ can also be defined as follows. Note that an isomorphism of marked surfaces with binary $\surf_{\vot}\to\surf'_{\vot'}$ induces an isomorphism of abelian groups $H_1(\surf_{\vot})\to H_1(\surf'_{\vot'})$. Given a flat surface $(C,D,\varphi)$, we have the map
\begin{equation}
\int : H_1(\gmsx(C,D,\varphi))\to\mathbb C, \quad \wg\mapsto \int_{\wg}\sqrt{\varphi}.
\end{equation}
Thus, if $(C,D,\varphi,f)\in\mathrm{FQuad}(\gmsx)$, we obtain then a map $H_1(\surf_{\vot})\to\mathbb C$. These maps combine to give a continuous map
\begin{equation}
\label{period map}
\prod\nolimits_{\vot}: \mathrm{FQuad}(\gmsx)\to\mathrm{Hom}(H_1(\gmsx),\mathbb C),
\end{equation}
called the \emph{period map} on $\mathrm{FQuad}(\gmsx)$. Based on the observation \eqref{relation-moduli} and \cite[Theorem 2.1]{HKK}, it turns out that $\prod_{\vot}$ is a local homeomorphism, which can be used to define an affine structure on $\mathrm{FQuad}(\gmsx)$ by pulling back the one on $\mathrm{Hom}(H_1(\gmsx),\mathbb C)$. It is worth noting that the $\mathbb C$-action on $\mathrm{FQuad}(\surf^{\lambda})$ can be induced on $\mathrm{FQuad}(\gmsx)$.

Let $\gmsx$ be the fixed GMSb, $\mathbb A^{\ast}$ be the fixed full formal closed arc system, and $\sg$ be the graded skew-gentle algebra arising from $\mathbb A^{\ast}$ on $\gmsx$, both of which are determined in \Cref{S-graph}. Recall that we denote by $\mathrm{Stab}(\mathcal D^b(\sg))$ the space of stability conditions on $\mathcal D^b(\sg)$. The \emph{principal component} of $\mathrm{Stab}(\mathcal D^b(\sg))$ corresponding to $\mathrm{EG}^{\circ}(\gmsx)$ as defined in Definition \ref{def:exchange graph}, is given by
\[
\mathrm{Stab}^{\circ}(\mathcal D^b(\sg)):=\mathbb C\cdot\bigcup_{\mathcal H\in\mathrm{EG}^{\circ}(\gmsx)}\mathrm{U}_{\vot}(\mathcal H).
\]


We now upgrade the isomorphism $X$ in \Cref{main} into an injection between complex manifolds. The image of this injecetion is the principal component $\mathrm{Stab}^{\circ}(\mathcal D^b(\sg))$ which turns out to be a generic-finite component corresponding to $\mathrm{EG}^{\circ}(\gmsx)$. The proof follows closely the arguments in \cite{BS,BMQS,CHQ}, with some modifications due to the presence of binary, which we will highlight.

It is known that there exists a one-to-one correspondence between the connected components of $\mathrm{FQuad}(\surf^{\lambda})$ and those of $\mathrm{EG}(\surf^{\lambda})$. Here, we can also identify a connected component $\mathrm{FQuad}^{\circ}(\gmsx)\subset\mathrm{FQuad}(\gmsx)$ corresponding to $\mathrm{EG}^{\circ}(\gmsx)$. The stratification of the moduli space $\mathrm{FQuad}(\surf^{\lambda})$ has already been given in \cite{CHQ}. The relation in Equation \eqref{relation-moduli} will induce a stratification on $\mathrm{FQuad}^{\circ}(\gmsx)$, given by
\[
B_0^{\circ}=B_1^{\circ}\subset B_2^{\circ}\subset\cdots\subset\mathrm{FQuad}^{\circ}(\gmsx)
\]
where the subspace
\[
B_p^{\circ}:=B_p^{\circ}(\gmsx)=\{q\in\mathrm{FQuad}^{\circ}(\gmsx)\,|\,r_q+2s_q\le p\}
\]
is defined by the number $s_q$ of saddle trajectories and $r_q$ of recurrent trajectories. We define $F_p^{\circ}=B_p^{\circ}\backslash B_{p-1}^{\circ}$.

As in \cite[Lemma 5.2]{CHQ}, we can show that this stratification satisfies the ascending chain property, i.e., there exists a positive integer $n$ such that $B_p^{\circ}=\mathrm{FQuad}^{\circ}(\gmsx)$ if $p\geq n$. Note that the set of saddle trajectories is linearly independent in $H_1(\gmsx)$ and, hence $s_q\le\mathrm{dim}H_1(\gmsx)$. Moreover, each recurrent trajectory is bounded by a ring domain or some (at least one) saddle trajectories. In addition, each saddle trajectory can appear in at most two boundaries of recurrent trajectories. Therefore, for any  $q\in\mathrm{FQuad}^{\circ}(\gmsx)$, we have $r_q+2s_q\le4\mathrm{dim}H_1(\gmsx)$.

Since there are only finitely many points in $\M$ and at most countably many compact arcs that are saddle trajectories, we can conclude that $\mathrm{FQuad}^{\circ}(\gmsx)=\mathbb C\cdot B_0^{\circ}(\gmsx)$.

Recall that there is the period map \eqref{period map} which defines a local homeomorphism. We can now state the following result, which is similar to \cite[Proposition 5.3]{CHQ}, and the proof relies on the surface arguments from \cite[\S 7.2]{BMQS}.

\begin{proposition}
\label{wall has ends}
For $p>2$, each component of the stratum $F_p^{\circ}$ contains a point $q$ and a neighbourhood $U\subset\mathrm{FQuad}^{\circ}(\gmsx)$ of $q$ such that $U\cap B_p^{\circ}$ is contained in the locus $\int_{\wa}\sqrt{q}\in\mathbb R$ for some $\wa\in H_1(\gmsx)$. Moreover, this containment is strict in the sense that $U\cap B_{p-1}^{\circ}$ is connnected.

As a result, we conclude that any path in $\mathrm{FQuad}^{\circ}(\gmsx)$ is homotopic relative to its endpoints to a path in $B_2^{\circ}$.
\end{proposition}

Now, we can state the following main theorem in this section.
\begin{theorem}
\label{main2}
There exists an injection $\iota$ between the moduli spaces
\[
\iota:\mathrm{FQuad}^{\circ}(\gmsx)\to\mathrm{Stab}(\mathcal D^b(\sg)),
\]
which is extended by the isomorphism $X$ in \Cref{main}. In particular, the image of $\iota$ is $\mathrm{Stab}^{\circ}(\mathcal D^b(\sg))$, which turns out to be the generic-finite component corresponding to $\mathrm{EG}^{\circ}(\gmsx)$.
\end{theorem}
\begin{proof}
We will provide a sketch of the argument, which closely follows the work in \cite{BS,BMQS,CHQ}. The Grothendieck group can be identified with the homology group as $\xi:K(\mathcal D^b(\sg))\to H_1(\gmsx)$ by Lemma \ref{K0 vs H1}. Next, we define a map $\iota$ from the saddle-free part $B_0^{\circ}(\gmsx)$ of $\mathrm{FQuad}^{\circ}(\gmsx)$ to $\mathrm{Stab}(\mathcal D^b(\sg))$. Suppose we have an $\gmsx$-framed quadratic differential $(C,D,\varphi,f)$ in $B_0^{\circ}(\gmsx)$ and an associated S-graph $\mathbb S$ determined by all saddle connections. We then map it to the stability condition $\sigma$ determined by the heart $ \mathcal{H}_{\mathbb{S}}=\langle X_{\CA}|\,\CA\in\CAS\rangle$ with central charge given by the period map in \Cref{period map}, i.e. $Z=\int\circ\xi$. We then extend the map $\iota$ from $B_0^{\circ}$ to $B_2^{\circ}$ continuously by compatible $\mathbb C$-actions on both sides. Finally, we extend the map $\iota$ from $B_p^{\circ}$ to $B_{p+1}^{\circ}$ inductively, using the same argument as \cite[Proposition 5.8]{BS}, \cite[\S 7.2]{BMQS}, and the walls-have-ends property on any connected component of $F_p^{\circ}$, where Proposition \ref{wall has ends} is essential. The ascending chain property of this stratification ensures that the inductive process terminates after a finite number of steps.

By the construction above, we observe that the image of $\iota$ coincides precisely with the principal component $\mathrm{Stab}^{\circ}(\mathcal D^b(\sg))$. Furtherover, since $\iota$ is both open and closed, $\mathrm{Stab}^{\circ}(\mathcal D^b(\sg))$ becomes a generic-finite component. Therefore, we obtain an injection between the moduli spaces $\mathrm{FQuad}^{\circ}(\gmsx)$ and $\mathrm{Stab}(\mathcal D^b(\sg))$, whose image is the generic-finite component $\mathrm{Stab}^{\circ}(\mathcal D^b(\sg))$ corresponding to $\mathrm{EG}^{\circ}(\gmsx)$.
\end{proof}


\clearpage

\appendix

\section{Silting arc systems and flips}\label{appendix-1}

Dually to the S-graphs and their flips, we can define silting arc systems (or called mixed-angulations, cf. \cite{CHQ}) and their flips, which generalizes the gentle case (cf. \cite{CS}).

Let $\mathbb{T}=\{\wg_1,...,\wg_n\}$ be a full formal open arc system on $\gmsx$. We call $\mathbb{T}$ a \emph{silting arc system (SAS)} if all of the intersection indices of $\wg_1,...,\wg_n$ are less than 1.

For a collection $\mathbb{T}=\{\wg_1,...,\wg_n\}$ of admissible open arcs on $\gmsx$, let $X_{\mathbb{T}}:=\oplus_{i=1}^nX_{\wg_i}$, where $X_{\wg_i}$ is the indecomposable object in ${\rm per}\,\sg$ corresponding to the open arc $\wg_i$ by the duality of \Cref{thm:int formula} (cf. \cite{QZZ}). In fact, a collection $\mathbb{T}$ of admissible open arcs is a SAS if and only if $X_{\mathbb{T}}$
is a silting object in ${\rm per}\,\sg$. Moreover, if $X$ is a basic silting arc object, then $X$ is isomorphic to $\oplus_{i=1}^nX_{\wg_i}$, where $\mathbb{T}=\{\wg_1,...,\wg_n\}$ is a SAS on $\gmsx$.

Now we give the rule of flips of SASs. Let $\mathbb{T}$ be a SAS and $\wg$ be a graded admissible open arc in $\mathbb{T}$. Let $M_1$ and $M_2$ be the endpoints of $\wg$, which may coincide. Consider the next arcs of $\wg$ anti-clockwisely in the $\mathbb{T}$-polygons satisfying that the intersection indices from $\wg$ to the next arcs at $M_1$ or $M_2$ are 0. We denote the next arcs by $\wg_1$ and $\wg_2$ respectively if they exist.

\begin{itemize}
    \item The usual case: we define $\wg^{\sharp}$ by moving $M_i$ along $\wg_i$ ($i=1,2$) simultaneously and the grading of $\wg^{\sharp}$ satisfies that the intersection indices from $\wg_i$ to $\wg^{\sharp}$ at $M_i$ are 0.
    \item Special case 1: we define $\wg^{\sharp}=\wg[1]$ if neither of $\wg_1$ and $\wg_2$ exists.
    \item Special case 2: $\wg_i$ has a twin $\wg_i'$ in $\mathbb{T}$ with the same degree for some $i=1,2$. That is, ${\rm ind}_{M_i}(\wg,\wg_i)=0$ and ${\rm ind}_{M_i}(\wg,\wg_i')=0$. Then we define $\widetilde{\Gamma}_i$ as the smoothing out of $\wg_i$ and $\wg'_i$ at $M'_i$ (which inherits the grading from one of $\wg_i$ and $\wg_i'$) and replace $\wg_i$ by $\widetilde{\Gamma}_i$ instead.
\end{itemize}

We define $\mu^{\sharp}_{\wg}(\wg)=\wg^{\sharp}$ and $\mu^{\sharp}_{\wg}(\alpha)=\alpha$ for any $\alpha\in\mathbb{T}\backslash\{\wg\}$. Moreover, define $\mu^{\sharp}_{\wg}(\mathbb{T}):=\{\mu^{\sharp}_{\wg}(\alpha)|\,\alpha\in\mathbb{T}\}$, called the \emph{forward flip} of $\mathbb{T}$ via $\wg$. Dually, we can define the \emph{backward flip} $\mu^{\flat}_{\wg}(\mathbb{T})$ of a SAS $\mathbb{T}$ via $\wg$. The properties of flips of SASs are dual to those of S-graphs.

\begin{figure}[h]\centering
	\def\dexc{black!10!blue!50!green}
	\begin{tikzpicture}[scale=0.6]
	\draw[very thick,blue] (-2,2)--(-2,-2)--(2,-2)--(2,2)--cycle;
	\draw[very thick,cyan] (-2,2)--(2,-2);
	\draw[very thick,blue] (-2,-2)--(2,2);
	\draw[blue] (-2,2)\nn;
	\draw[blue] (-2,-2)\nn;
	\draw[blue] (2,2)\nn;
	\draw[blue] (2,-2)\nn;
    \draw (-2,2)node[above]{$M_2'$};
    \draw (-2,0)node[left]{$\t{\alpha}_2$};
    \draw (-2,-2)node[below]{$M_1$};
    \draw (0,-2)node[below]{$\t{\gamma}_1$};
    \draw (2,-2)node[below]{$M_1'$};
    \draw (-0.4,1)node{$\t{\gamma}^{\sharp}$};
    \draw (2,2)node[above]{$M_2$};
    \draw (2,0)node[right]{$\t{\alpha}_1$};
    \draw (0.5,1)node{$\t{\gamma}$};
    \draw (0,2)node[above]{$\t{\gamma}_2$};
    
    \begin{scope}[shift={(7,0)}]
	\draw[very thick,cyan] (0,2)--(0,-2);
	\draw[blue] (0,2)\nn;
	\draw[blue] (0,-2)\nn;
    \draw (0,2)node[above]{$M_2'=M_2$};
    \draw (0,-2)node[below]{$M_1'=M_1$};
    \draw (0,0)node[right]{$\t{\gamma}^{\sharp}=\t{\gamma}[1]$};
    \end{scope}
    
    \begin{scope}[shift={(14,5)}, scale=0.5]
    \draw[ultra thick]plot [smooth,tension=1] coordinates {(3,-4.5) (-90:5.2) (-3,-4.5)};
	\draw[ultra thick,fill=gray!10] (-90:11) ellipse (1.5);
	\draw (-90:11) node {$\Vot$};
	\draw[very thick,blue] (-90:5.2) .. controls +(-45:6) and +(-30:5) .. (-90:12.5);
	\draw[very thick,blue] (-90:5.2) .. controls +(-135:6) and +(-150:5) .. (-90:12.5);
	\draw[very thick,blue] (-90:5.2) .. controls +(-15:3) and +(130:3) .. (-60:11);
	\draw[very thick,cyan] (-60:11) .. controls +(-110:13) and +(-140:13) ..	(-90:5.2);
	\draw[blue] (-90:12.5)\nn;
	\draw[blue] (-90:5.2)\nn;
	\draw[blue] (-60:11)\nn;
    \draw (-90:9.5)\DEC;
    \draw (3.9,-11) node {$\wg_1$} ;
    \draw (-2.4,-11) node {$\wg_1'$} ;
    \draw (4,-6.8) node {$\wg$} ;
    \draw (3,-15) node {$\wg^{\sharp}$};
    \end{scope}
    
    \end{tikzpicture}
    \caption{Forward flips of SASs}\label{fig:flip2}
\end{figure}

\clearpage

\section{Classification of the smoothing out}\label{appendix-2}

In order to finish the induction in \Cref{classify}, we need to show that $T^{\sharp}$ is an arc object. But we have shown in \Cref{classify} that the homology class $[(\CA_T)^{\sharp}_{\CA_S}]$ corresponds to the class $[T^{\sharp}]$ in the Grothendieck group. Now we only need to show that the smoothing out $(\CA_T)^{\sharp}_{\CA_S}$ is an unknotted arc using case by case study.

We will list all the cases of $\CA_S$ and $\CA_T$. The cases where $\CA_S$ or $\CA_T$ has a twin in the S-graph $\mathbb{S}$ are called binary cases, while the other cases are called binary-free cases.

\textbf{Binary-free cases} (cf. \Cref{fig:001}, where the cyan arc in each figure is $(\CA_T)_{\CA_S}^{\sharp}$):

There are three cases, depending on the number of endpoints of $\CA_S$ and $\CA_T$.

\begin{itemize}
    \item[\textbf{F1:}] the number of endpoints of $\CA_S$ and $\CA_T$ is 3. In this case, $|\partial\CA_S|=|\partial\CA_T|=2$ and they share exactly one common endpoint, cf. Figure 1 of \Cref{fig:001}.
    \item[\textbf{F2:}] the number of endpoints of $\CA_S$ and $\CA_T$ is 2. There are three subcases, depending on $|\partial\CA_S|$ and $|\partial\CA_T|$.
    \begin{itemize}
        \item[\textbf{F2.1:}] $|\partial\CA_S|=|\partial\CA_T|=2$. In this subcase, $\CA_T$ and $\CA_S$ share two endpoints. There are three subcases, depending on the oriented intersections from $\CA_T$ to $\CA_S$.
        \begin{itemize}
            \item[\textbf{F2.1.1:}] there is only one oriented intersection from $\CA_T$ to $\CA_S$, cf. Figure 2.1.1 of \Cref{fig:001}, where $d\geq 1$.
            \item[\textbf{F2.1.2:}] there are two oriented intersections from $\CA_T$ to $\CA_S$, only one of which is of index 1, cf. Figure 2.1.2 of \Cref{fig:001}, where $d\geq 2$.
            \item[\textbf{F2.1.3:}]there are two oriented intersections from $\CA_T$ to $\CA_S$ of index 1, cf. Figure 2.1.3 of \Cref{fig:001}.
        \end{itemize}
        \item[\textbf{F2.2:}] $|\partial\CA_T|=1$ and $|\partial\CA_S|=2$. In this subcase, $\CA_T$ and $\CA_S$ share exactly one common endpoint. There are two subcases, depending on the order of the three arc segments at the common endpoint.
        \begin{itemize}
            \item[\textbf{F2.2.1:}] the order of the three arc segments at the common endpoint is $\CA_T,\CA_S,\CA_T$ in clockwise, cf. Figure 2.2.1 of \Cref{fig:001}, where $d\geq 1$.
            \item[\textbf{F2.2.2:}] the order of the three arc segments at the common endpoint is $\CA_T,\CA_T,\CA_S$ in clockwise, cf. Figure 2.2.2 of \Cref{fig:001}, where $d\geq 2$.
        \end{itemize}
        \item[\textbf{F2.3:}] $|\partial\CA_T|=2$ and $|\partial\CA_S|=1$. In this subcase, $\CA_T$ and $\CA_S$ share exactly one common endpoint. There are two subcases, depending on the order of the three arc segments at the common endpoint.
        \begin{itemize}
            \item[\textbf{F2.3.1:}] the order of the three arc segments at the common endpoint is $\CA_S,\CA_T,\CA_S$ in clockwise, cf. Figure 2.3.1 of \Cref{fig:001}, where $d\geq 1$.
            \item[\textbf{F2.3.2:}] the order of the three arc segments at the common endpoint is $\CA_T,\CA_S,\CA_S$ in clockwise, cf. Figure 2.3.2 of \Cref{fig:001}, where $d\geq 2$.
        \end{itemize}
    \end{itemize}
    \item[\textbf{F3:}] the number of endpoints of $\CA_S$ and $\CA_T$ is 1. In this case, all endpoints of $\CA_T$ and $\CA_S$ coincide. There are five subcases, depending on the order of the four arc segments at the common endpoint.
    \begin{itemize}
        \item[\textbf{F3.1:}] the order of the four arc segments at the common endpoint is $\CA_T,\CA_T,\CA_S,\CA_S$ in clockwise, cf. Figure 3.1 of \Cref{fig:001}, where $c,d\geq 2$.
        \item[\textbf{F3.2:}] the order of the four arc segments at the common endpoint is $\CA_T,\CA_S,\CA_S,\CA_T$ in clockwise, cf. Figure 3.2 of \Cref{fig:001}, where $c,d\geq 2$.
        \item[\textbf{F3.3:}] the order of the four arc segments at the common endpoint is $\CA_S,\CA_T,\CA_T,\CA_S$ in clockwise, cf. Figure 3.3 of \Cref{fig:001}, where $c,d\geq 2$.
        \item[\textbf{F3.4:}] the order of the four arc segments at the common endpoint is $\CA_T,\CA_S,\CA_T,\CA_S$ clockwise, cf. Figure 3.4 of \Cref{fig:001}, where $c,c',d\geq 1$. But there are three subcases, depending on whether $c=1$ and whether $c'=1$.
        \begin{itemize}
            \item[\textbf{F3.4.1:}] $c=1$ but $c'\neq 1$, cf. Figure 3.4.1 of \Cref{fig:001}.
            \item[\textbf{F3.4.2:}] $c'=1$ but $c\neq 1$, cf. Figure 3.4.2 of \Cref{fig:001}.
            \item[\textbf{F3.4.3:}] $c=c'=1$, cf. Figure 3.4.3 of \Cref{fig:001}.
        \end{itemize}
        \item[\textbf{F3.5:}] the order of the four arc segments at the common endpoint is $\CA_S,\CA_T,\CA_S,\CA_T$ clockwise, cf. Figure 3.5 of \Cref{fig:001}, where $c,d\geq 1$.
    \end{itemize}
\end{itemize}

\textbf{Binary cases} (cf. \Cref{fig:2}, where the cyan arc in each figure is $(\CA_T)_{\CA_S}^{\sharp}$):

Since $T\neq S$ and $\mathrm{Ext}^1(T,S)\neq 0$, $\CA_T$ and $\CA_S$ are not twins by \Cref{thm:int formula}. Moreover, by the simple tilting formula, any two of $\CA_T$, $\CA_S$ and $\CA_{T^{\sharp}}$ are not twins.

There are two binary cases, depending on whether $\CA_S$ and $\CA_T$ has a twin in $\CAS$. 

\begin{itemize}
    \item[\textbf{B1:}] Exactly one of $\CA_S$ and $\CA_T$ has a twin in $\CAS$. Now we only classify the subcases where the twin $\CA'_S$ of $\CA_S$ is in $\CAS$ but $\CA_T$ has no twin in $\CAS$ and omit the dual subcases. We denote the endpoint of $\CA_S$ on the binary by $P$ and the other endpoint by $M$.
    
    Note that for some subcase below, it has a twin subcase, in which all is the same as the original subcase except the locations of $\CA_S$ and $\CA_S'$. One subcase can become its twin subcase by taking ${\rm D}^{-2}_{\Vot}$-action around some binary $\Vot$. Such twin subcases will be omitted.
    \begin{itemize}
        \item[\textbf{B1.1:}] the number of endpoints of $\CA_S$ and $\CA_T$ is 3. In this subcase, $|\partial\CA_T|=2$ and there is only one common endpoint of $\CA_S$ and $\CA_T$. There are two subsubcases, depending on that the common endpoint of $\CA_S$ and $\CA_T$.
        \begin{itemize}
            \item[\textbf{B1.1.1:}] the common endpoint of $\CA_S$ and $\CA_T$ is $M$, cf. Figure 1.1.1 of \Cref{fig:2}.
            \item[\textbf{B1.1.2:}] the common endpoint of $\CA_S$ and $\CA_T$ is $P$, cf. Figure 1.1.2 of \Cref{fig:2}.
        \end{itemize}
        \item[\textbf{B1.2:}] the number of endpoints of $\CA_S$ and $\CA_T$ is 2. There are two subsubcases, depending on $|\partial\CA_T|$.
        \begin{itemize}
            \item[\textbf{B1.2.1:}] $|\partial\CA_T|=2$. In this subsubcase, $\CA_T$ and $\CA_S$ share two common endpoints $M$ and $P$. There are three orders of arc segments of $\CA_T$, $\CA_S$ and $\CA_S'$ at $M$.
            \begin{itemize}
                \item[\textbf{B1.2.1.1:}] the order of the three arc segments at $M$ is $\CA_T, \CA_S, \CA'_S$ clockwise, cf. Figure 1.2.1.1 of \Cref{fig:2}, where $d\geq 2$.
                \item[\textbf{B1.2.1.2:}] the order of the three arc segments at $M$ is $\CA_T, \CA_S', \CA_S$ clockwise. Then there are two oriented intersections from $\CA_T$ to $\CA_S$. But since the angle between $\CA_S'$ and $\CA_S$ is cut out by $\CA_T$, $\CA_S'$ and $\CA_S$ are not identical twins by $1^{\circ}$ of Lemma \ref{lemma:index} and \Cref{thm:int formula}. And thus the oriented intersection from $\CA_T$ to $\CA_S$ at $M$ can not have index 1, cf. Figure 1.2.1.2 of \Cref{fig:2}, where $d\geq 2$.
                \item[\textbf{B1.2.1.3:}] the order of the three arc segments at $M$ is $\CA_S', \CA_T, \CA_S$ clockwise, cf. Figure 1.2.1.3 of \Cref{fig:2}, where $d\geq 2$.
            \end{itemize}
            \item[\textbf{B1.2.2:}] $|\partial\CA_T|=1$. In this subsubcase, $\CA_S$ and the two segments of $\CA_T$ share a common endpoint. Then the common endpoint may be $M$ or $P$, and there are three orders of arc segments at the common endpoint.
            \begin{itemize}
                \item[\textbf{B1.2.2.1:}] the common endpoint of $\CA_S$ and $\CA_T$ is $M$, and the order of the four arc segments at $M$ is $\CA_T,\CA_S,\CA'_S,\CA_T$ (or $\CA_T,\CA'_S,\CA_S,\CA_T$) clockwise, cf. Figure 1.2.2.1 of \Cref{fig:2}, where $d\geq 1$.
                \item[\textbf{B1.2.2.2:}] the common endpoint of $\CA_S$ and $\CA_T$ is $M$, and the order of the four arc segments at $M$ is $\CA_T,\CA_T,\CA_S,\CA'_S$ (or $\CA_T,\CA_T,\CA'_S,\CA_S$) clockwise, cf. Figure 1.2.2.2 of \Cref{fig:2}, where $d\geq 2$.
                \item[\textbf{B1.2.2.3:}] the common endpoint of $\CA_S$ and $\CA_T$ is $P$, cf. Figure 1.2.2.3 of \Cref{fig:2}, where $d\geq 2$.
            \end{itemize}
        \end{itemize}
    \end{itemize}
    \item[\textbf{B2:}] the twin $\CA_S'$ of $\CA_S$ and the twin $\CA_T'$ of $\CA_T$ are both in $\CAS$, cf. \Cref{fig:2}. In this case, there is only one common endpoint of the four arcs $\CA_S$, $\CA_S'$, $\CA_T$ and $\CA_T'$. We denote the common endpoint by $M$. There are three subcases, depending on whether $M$ is on a binary and the order of the four arc segments at $M$.
    \begin{itemize}
        \item[\textbf{B2.1:}] the common endpoint $M$ is not on a binary, cf. Figure 2.1 of \Cref{fig:2}. In this subcase, there are four possible order of the four arc segments at $M$, but they can become each other by ${\rm D}^2_{\Vot}$-action.
        \item[\textbf{B2.2:}] the common endpoint $M$ is on a binary, and the order of the four arc segments at $M$ is $\CA_T,\CA_S,\CA_S',\CA_T'$ (or $\CA_T,\CA_S',\CA_S,\CA_T'$), cf. Figure 2.2 of \Cref{fig:2}.
        \item[\textbf{B2.3:}] the common endpoint $M$ is on a binary, and the order of the four arc segments at $M$ is $\CA_S',\CA_T',\CA_T,\CA_S$ (or $\CA_S',\CA_T,\CA_T',\CA_S$), cf. Figure 2.3 of \Cref{fig:2}.
    \end{itemize}
\end{itemize}

In each case above, we see that the smoothing out $(\CA_T)_{\CA_S}^{\sharp}$ is an unknotted arc, and thus $T^{\sharp}$ is an arc object.

\begin{figure}[htbp]
	\begin{tikzpicture}[scale=.3]
	\draw[ultra thick]plot [smooth,tension=1] coordinates {(3,-4.5) (-90:5.2) (-3,-4.5)};
	\draw[red] (-90:5.2) .. controls +(-120:3) and +(30:3) .. (-5,-10.5);
	\draw[red] (-90:5.2) .. controls +(-60:3) and +(150:3) .. (5,-10.5);
	\draw[ultra thick]plot [smooth,tension=1] coordinates {(5.4,-12) (5,-10.5) (5.4,-9)};
	\draw[ultra thick]plot [smooth,tension=1] coordinates {(-5.4,-12) (-5,-10.5) (-5.4,-9)};
    \draw[cyan, ultra thick] (-5,-10.5) .. controls +(20:3) and +(160:3) .. (5,-10.5);
    \draw (-90:5.2)\DEC;
    \draw (5,-10.5)\DEC;
    \draw (-5,-10.5)\DEC;
	\draw (-90:3) node{Figure 1};
    \draw (-90:6.5) node {1} ;
    \draw (3.2,-8) node {$\CA_T$} ;
    \draw (-3.2,-8) node {$\CA_S$} ;

    \begin{scope}[shift={(13,0)}]
	\draw[ultra thick]plot [smooth,tension=1] coordinates {(3,-4.5) (-90:5.2) (-3,-4.5)};
	\draw[ultra thick]plot [smooth,tension=1] coordinates {(3,-13.7) (-90:13) (-3,-13.7)};
	\draw[red] (-90:5.2) .. controls +(-45:5) and +(45:5) .. (-90:13);
	\draw[red] (-90:5.2) .. controls +(-135:5) and +(135:5) .. (-90:13);
    \draw[cyan, ultra thick] (0,-13)
        .. controls +(60:3) and +(0:1) .. (0,-8)
        .. controls +(180:1) and +(120:3) .. (0,-13);
    \draw (-90:13)\DEC;
	\draw (-90:5.2)\DEC;
    \draw (-90:3) node{Figure 2.1.1};
    \draw (-90:11.5) node {$d$} ;
    \draw (-90:6.75) node {1} ;
    \draw (3.2,-11) node {$\CA_T$} ;
    \draw (-3.2,-11) node {$\CA_S$} ;
    \end{scope}

	\begin{scope}[shift={(26,0)}]
	\draw[ultra thick]plot [smooth,tension=1] coordinates {(3,-4.5) (-90:5.2) (-3,-4.5)};
	\draw[ultra thick,fill=gray!10] (-90:11) ellipse (1.5);
	\draw[red] (-90:5.2) .. controls +(-45:6) and +(-30:5) .. (-90:12.5);
	\draw[red] (-90:5.2) .. controls +(-135:6) and +(-150:5) .. (-90:12.5);
    \draw[cyan, ultra thick] (0,-12.5)
        .. controls +(-10:3) and +(0:3) .. (0,-8)
        .. controls +(180:3) and +(-170:3) .. (0,-12.5);
    \draw (-90:12.5)\DEC;
	\draw (-90:5.2)\DEC;
    \draw[blue] (-90:9.5)\nn;
    \draw (-90:3) node{Figure 2.1.2};
    \draw (-90:13.8) node {$d$} ;
    \draw (-90:6.25) node {1} ;
    \draw (4,-11) node {$\CA_T$} ;
    \draw (-4,-11) node {$\CA_S$} ;
    \end{scope}

    \begin{scope}[shift={(39,0)}]
	\draw[ultra thick]plot [smooth,tension=1] coordinates {(3,-4.5) (-90:5.2) (-3,-4.5)};
	\draw[ultra thick,fill=gray!10] (-90:11) ellipse (1.5);
	\draw[red] (-90:5.2) .. controls +(-45:6) and +(-30:5) .. (-90:12.5);
	\draw[red] (-90:5.2) .. controls +(-135:6) and +(-150:5) .. (-90:12.5);
    \draw[cyan, ultra thick] (-90:5.2)
        .. controls +(-145:9) and +(160:6) .. (0,-15)
        .. controls +(-20:4) and +(0:4) .. (0,-8)
        .. controls +(-170:3) and +(-170:3) .. (0,-12.5);
    \draw (-90:12.5)\DEC;
	\draw (-90:5.2)\DEC;
    \draw[blue] (-90:9.5)\nn;
    \draw (-90:3) node{Figure 2.1.3};
    \draw (-90:13.5) node {1} ;
    \draw (-90:6.25) node {1} ;
    \draw (4,-11) node {$\CA_T$} ;
    \draw (-4,-11) node {$\CA_S$} ;
    \end{scope}

    \begin{scope}[shift={(0,-13.5)}]
	\draw[ultra thick]plot [smooth,tension=1] coordinates {(3,-4.5) (-90:5.2) (-3,-4.5)};
	\draw[ultra thick,fill=gray!10] (-90:11) ellipse (1.5);
	\draw[red] (-90:5.2) .. controls +(-45:4) and +(90:1) .. (3.3,-11.5)
	.. controls +(-90:2) and +(0:1) .. (-90:14)
	.. controls +(180:1) and +(-90:2) .. (-3.3,-11.5)
	.. controls +(90:1) and +(-135:4) .. (-90:5.2);
	\draw[red] (-90:5.2) -- (-90:9.5);
	\draw[cyan, ultra thick] (-90:5.2)
        .. controls +(-130:4) and +(180:3) .. (0,-13.5)
        .. controls +(0:3) and +(60:4) .. (0,-9.5);
    \draw[blue] (-90:12.5)\nn;
	\draw (-90:5.2)\DEC;
    \draw (-90:9.5)\DEC;
    \draw (-90:3) node{Figure 2.2.1};
    \draw (-94:6.75) node {$d$} ;
    \draw (-86:6.75) node {1} ;
    \draw (4,-11) node {$\CA_T$} ;
    \draw (-0.7,-8.5) node {$\CA_S$} ;
    \end{scope}

    \begin{scope}[shift={(13,-13.5)}]
	\draw[ultra thick]plot [smooth,tension=1] coordinates {(3,-4.5) (-90:5.2) (-3,-4.5)};
	\begin{scope}[rotate around = {35:(-90:5.2)}]
	\draw[red] (-90:5.2) .. controls +(-60:4) and +(90:1) .. (3,-11.5)
	.. controls +(-90:2) and +(0:1) .. (-90:14)
	.. controls +(180:1) and +(-90:2) .. (-3,-11.5)
	.. controls +(90:1) and +(-120:4) .. (-90:5.2);
	\end{scope}
	\draw[cyan, ultra thick] (-5,-10.5)
        .. controls +(0:4) and +(180:4) .. (3,-14)
        .. controls +(0:6) and +(-10:9) .. (0,-5.2);
	\draw[rotate around = {35:(-90:5.2)}] (-90:7) node {$d$} ;
	\draw[red] (-90:5.2) .. controls +(-120:3) and +(30:3) .. (-5,-10.5);
	\draw[ultra thick]plot [smooth,tension=1] coordinates {(-5.4,-12) (-5,-10.5) (-5.4,-9)};
    \draw (-5,-10.5)\DEC;
	\draw (-90:5.2)\DEC;
    \draw (-90:3) node{Figure 2.2.2};
    \draw (-0.3,-7) node {1} ;
    \draw (4,-12) node {$\CA_T$} ;
    \draw (-3.5,-8.3) node {$\CA_S$} ;
    \end{scope}

    \begin{scope}[shift={(26,-13.5)}]
	\draw[ultra thick]plot [smooth,tension=1] coordinates {(3,-4.5) (-90:5.2) (-3,-4.5)};
	\draw[ultra thick,fill=gray!10] (-90:11) ellipse (1.5);
	\draw[red] (-90:5.2) .. controls +(-45:4) and +(90:1) .. (3.3,-11.5)
	.. controls +(-90:2) and +(0:1) .. (-90:14)
	.. controls +(180:1) and +(-90:2) .. (-3.3,-11.5)
	.. controls +(90:1) and +(-135:4) .. (-90:5.2);
	\draw[cyan, ultra thick] (-90:5.2)
        .. controls +(-50:4) and +(0:3) .. (0,-13.5)
        .. controls +(180:3) and +(120:4) .. (0,-9.5);
	\draw[red] (-90:5.2) -- (-90:9.5);
    \draw[blue] (-90:12.5)\nn;
	\draw (-90:5.2)\DEC;
    \draw (-90:9.5)\DEC;
    \draw (-90:3) node{Figure 2.3.1};
    \draw (-95:6.75) node {1} ;
    \draw (-87:6.75) node {$d$} ;
    \draw (4,-11) node {$\CA_S$} ;
    \draw (0.8,-8.5) node {$\CA_T$} ;
    \end{scope}

    \begin{scope}[shift={(39,-13.5)}]
	\draw[ultra thick]plot [smooth,tension=1] coordinates {(3,-4.5) (-90:5.2) (-3,-4.5)};
	\begin{scope}[rotate around = {-35:(-90:5.2)}]
	\draw[red] (-90:5.2) .. controls +(-60:4) and +(90:1) .. (3,-11.5)
	.. controls +(-90:2) and +(0:1) .. (-90:14)
	.. controls +(180:1) and +(-90:2) .. (-3,-11.5)
	.. controls +(90:1) and +(-120:4) .. (-90:5.2);
	\end{scope}
	\draw[cyan, ultra thick] (5,-10.5)
        .. controls +(180:4) and +(0:4) .. (-3,-14)
        .. controls +(180:6) and +(-170:9) .. (0,-5.2);
	\draw[rotate around = {-35:(-90:5.2)}] (-90:7) node {$d$} ;
	\draw[ultra thick]plot [smooth,tension=1] coordinates {(5.4,-12) (5,-10.5) (5.4,-9)};
	\draw[red] (-90:5.2) .. controls +(-60:3) and +(150:3) .. (5,-10.5);
    \draw (5,-10.5)\DEC;
	\draw (-90:5.2)\DEC;
    \draw (-90:3) node{Figure 2.3.2};
    \draw (-87:7) node {1} ;
    \draw (-4,-12) node {$\CA_S$} ;
    \draw (3.5,-8) node {$\CA_T$} ;
    \end{scope}
	
	\begin{scope}[shift={(10,-27)}]
	\draw[ultra thick]plot [smooth,tension=1] coordinates {(3,-4.5) (-90:5.2) (-3,-4.5)};
	\begin{scope}[rotate around = {35:(-90:5.2)}]
	\draw[red] (-90:5.2) .. controls +(-60:4) and +(90:1) .. (3,-11.5)
	.. controls +(-90:2) and +(0:1) .. (-90:14)
	.. controls +(180:1) and +(-90:2) .. (-3,-11.5)
	.. controls +(90:1) and +(-120:4) .. (-90:5.2);
	\end{scope}
	\draw[rotate around = {35:(-90:5.2)}] (-90:7) node {$c$} ;
	\begin{scope}[rotate around = {-35:(-90:5.2)}]
	\draw[red] (-90:5.2) .. controls +(-60:4) and +(90:1) .. (3,-11.5)
	.. controls +(-90:2) and +(0:1) .. (-90:14)
	.. controls +(180:1) and +(-90:2) .. (-3,-11.5)
	.. controls +(90:1) and +(-120:4) .. (-90:5.2);
	\end{scope}
	\draw[rotate around = {-35:(-90:5.2)}] (-90:7) node {$d$} ;
	\draw[cyan, ultra thick] (0,-5.2)
        .. controls +(-170:9) and +(180:6) .. (-3,-14)
        .. controls +(0:1) and +(180:1) .. (3,-14)
        .. controls +(0:6) and +(-10:9) .. (0,-5.2);
	\draw (0,-9) node {1} ;
	\draw (-90:5.2)\DEC;
    \draw (-90:3) node{Figure 3.1};
    \draw (-4,-12) node {$\CA_S$} ;
    \draw (4,-12) node {$\CA_T$} ;
    \end{scope}

    \begin{scope}[shift={(23,-27)}]
    \draw[ultra thick]plot [smooth,tension=1] coordinates {(3,-4.5) (-90:5.2) (-3,-4.5)};
	\draw[red] (-90:5.2) .. controls +(-30:4) and +(90:1) .. (3.5,-11.5)
	.. controls +(-90:2) and +(0:1) .. (-90:14)
	.. controls +(180:1) and +(-90:2) .. (-3.5,-11.5)
	.. controls +(90:1) and +(-150:4) .. (-90:5.2);
	\begin{scope}[rotate around = {15:(-90:5.2)}]
	 \draw[red] (-90:5.2) .. controls +(-80:1) and +(90:1) .. (1,-10.5)
	.. controls +(-90:1) and +(0:0.5) .. (-90:12)
	.. controls +(180:0.5) and +(-90:1) .. (-1,-10.5)
	.. controls +(90:1) and +(-100:1) .. (-90:5.2);
	\draw (0,-7.5) node {$c$} ;
	\draw (0,-12.8) node {$\CA_S$} ;
	\end{scope}
	\begin{scope}[rotate around = {-20:(-90:5.2)}]
	\draw[cyan, ultra thick] (-90:5.2) .. controls +(-80:1) and +(90:1) .. (1,-10.5)
	.. controls +(-90:1) and +(0:0.5) .. (-90:12)
	.. controls +(180:0.5) and +(-90:1) .. (-1,-10.5)
	.. controls +(90:1) and +(-100:1) .. (-90:5.2);
	\end{scope}
	\draw (1.5,-7) node {1} ;
	\draw (-1,-7.5) node {$d$} ;
	\draw (-90:5.2)\DEC;
    \draw (-90:3) node{Figure 3.2};
    \draw (4,-13) node {$\CA_T$} ;
    \end{scope}

    \begin{scope}[shift={(36,-27)}]
    \draw[ultra thick]plot [smooth,tension=1] coordinates {(3,-4.5) (-90:5.2) (-3,-4.5)};
	\draw[red] (-90:5.2) .. controls +(-30:4) and +(90:1) .. (3.5,-11.5)
	.. controls +(-90:2) and +(0:1) .. (-90:14)
	.. controls +(180:1) and +(-90:2) .. (-3.5,-11.5)
	.. controls +(90:1) and +(-150:4) .. (-90:5.2);
	\begin{scope}[rotate around = {-15:(-90:5.2)}]
	\draw[red] (-90:5.2) .. controls +(-80:1) and +(90:1) .. (1,-10.5)
	.. controls +(-90:1) and +(0:0.5) .. (-90:12)
	.. controls +(180:0.5) and +(-90:1) .. (-1,-10.5)
	.. controls +(90:1) and +(-100:1) .. (-90:5.2);
	\draw (0,-7.5) node {$d$} ;
	\draw (0,-12.7) node {$\CA_T$} ;
	\end{scope}
	\begin{scope}[rotate around = {20:(-90:5.2)}]
	    \draw[cyan, ultra thick] (-90:5.2) .. controls +(-80:1) and +(90:1) .. (1,-10.5)
	.. controls +(-90:1) and +(0:0.5) .. (-90:12)
	.. controls +(180:0.5) and +(-90:1) .. (-1,-10.5)
	.. controls +(90:1) and +(-100:1) .. (-90:5.2);
	\end{scope}
	\draw (0.9,-7.5) node {$c$} ;
	\draw (-1.5,-7) node {1} ;
	\draw (-90:5.2)\DEC;
    \draw (-90:3) node{Figure 3.3};
    \draw (4,-13) node {$\CA_S$} ;
    \end{scope}
	
	\begin{scope}[shift={(0,-40.5)}]
	\draw[ultra thick]plot [smooth,tension=1] coordinates {(3,-4.5) (-90:5.2) (-3,-4.5)};
	\begin{scope}[rotate around = {20:(-90:5.2)}]
	\draw[red] (-90:5.2) .. controls +(-60:4) and +(90:1) .. (3,-11.5)
	.. controls +(-90:2) and +(0:1) .. (-90:14)
	.. controls +(180:1) and +(-90:2) .. (-3,-11.5)
	.. controls +(90:1) and +(-120:4) .. (-90:5.2);
	\end{scope}
	\draw[rotate around = {35:(-90:5.2)}] (-90:7) node {$c$} ;
	\filldraw[white, fill = white] (0,-13.2) circle (10pt);
	\begin{scope}[rotate around = {-20:(-90:5.2)}]
	\draw[red] (-90:5.2) .. controls +(-60:4) and +(90:1) .. (3,-11.5)
	.. controls +(-90:2) and +(0:1) .. (-90:14)
	.. controls +(180:1) and +(-90:2) .. (-3,-11.5)
	.. controls +(90:1) and +(-120:4) .. (-90:5.2);
	\end{scope}
	\draw[rotate around = {-35:(-90:5.2)}] (-90:7) node {$c'$} ;
	\draw (0,-8) node {$d$} ;
	\draw (-90:5.2)\DEC;
    \draw (-4,-14) node {$\CA_S$} ;
    \draw (4,-14) node {$\CA_T$} ;
        \draw (8,-10) node {$\Longleftrightarrow$};
            \draw (0,-3) node{Figure 3.4};
            \draw (16,-3) node{Figure 3.4.1};
            \filldraw[fill=gray!20] (12,-6) -- (20,-6) -- (20,-14) -- (16,-14) -- (16,-10) -- (12,-10) -- (12,-6);
    \foreach \i in {0,4} {
    \foreach \j in {0,4} {
    \draw (12+\i,-10+\j) -- (12+\i,-14+\j);
    \draw (12+\i,-10+\j) -- (16+\i,-10+\j);
    \draw[cyan, ultra thick] (16,-10) .. controls +(-100:3) and +(10:3) .. (12,-14);
    \begin{scope}[rotate around ={45:(16+\i,-10+\j)}]
    \filldraw[thick, fill=gray] (16+\i,-10+\j) .. controls +(145:0.5) and +(90:0.3) .. (14.5+\i,-10+\j)
	.. controls +(-90:0.3) and +(-145:0.5) .. (16+\i,-10+\j);
    \end{scope}
    }
    }
    \draw (20,-6) -- (20,-14);
    \draw (12,-14) -- (20,-14);
    \foreach \i in {0,4,8} {
    \foreach \j in {0,4,8} {
    \draw (12+\i,-14+\j)\DEC;
    }
    }
    \draw (18.5,-9.5) node {$\CA_T$};
    \draw (16.7,-7.5) node {$\CA_S$};
    \draw (16.5,-9.3) node {$d$};
    \draw (15.5,-9.3) node {1};
    \draw (16.5,-10.8) node {$c'$};
    \draw (22,-10) node {or};
    \begin{scope}[shift={(12,0)}]
    \filldraw[fill=gray!20] (12,-6) -- (20,-6) -- (20,-14) -- (16,-14) -- (16,-10) -- (12,-10) -- (12,-6);
            \foreach \i in {0,4} {
    \foreach \j in {0,4} {
    \draw (12+\i,-10+\j) -- (12+\i,-14+\j);
    \draw (12+\i,-10+\j) -- (16+\i,-10+\j);
    \draw[cyan, ultra thick] (16,-10) .. controls +(-170:3) and +(80:3) .. (12,-14);
    \begin{scope}[rotate around ={45:(16+\i,-10+\j)}]
    \filldraw[thick, fill=gray] (16+\i,-10+\j) .. controls +(145:0.5) and +(90:0.3) .. (14.5+\i,-10+\j)
	.. controls +(-90:0.3) and +(-145:0.5) .. (16+\i,-10+\j);
    \end{scope}
    }
    }
    \draw (20,-6) -- (20,-14);
    \draw (12,-14) -- (20,-14);
    \foreach \i in {0,4,8} {
    \foreach \j in {0,4,8} {
    \draw (12+\i,-14+\j)\DEC;
    }
    }
    \draw (18.5,-9.5) node {$\CA_T$};
    \draw (16.7,-7.5) node {$\CA_S$};
    \draw (16.5,-9.3) node {$d$};
    \draw (15.5,-9.3) node {$c$};
    \draw (16.5,-10.8) node {1};
    \end{scope}
    \draw (28,-3) node{Figure 3.4.2};
    \draw (34,-10) node{or};
    \draw (40,-3) node{Figure 3.4.3};
    \begin{scope}[shift={(24,0)}]
    \filldraw[fill=gray!20] (12,-6) -- (20,-6) -- (20,-14) -- (16,-14) -- (16,-10) -- (12,-10) -- (12,-6);
            \foreach \i in {0,4} {
    \foreach \j in {0,4} {
    \draw (12+\i,-10+\j) -- (12+\i,-14+\j);
    \draw (12+\i,-10+\j) -- (16+\i,-10+\j);
    \draw[cyan, ultra thick] (16,-10) .. controls +(-100:1) and +(60:2) .. (14,-14);
    \draw[cyan, ultra thick, dashed] (16,-6) .. controls +(-100:1) and +(60:2) .. (14,-10);
    \draw[cyan, ultra thick] (14,-10)--(12,-14);
    \begin{scope}[rotate around ={45:(16+\i,-10+\j)}]
    \filldraw[thick, fill=gray] (16+\i,-10+\j) .. controls +(145:0.5) and +(90:0.3) .. (14.5+\i,-10+\j)
	.. controls +(-90:0.3) and +(-145:0.5) .. (16+\i,-10+\j);
    \end{scope}
    }
    }
    \draw (20,-6) -- (20,-14);
    \draw (12,-14) -- (20,-14);
    \foreach \i in {0,4,8} {
    \foreach \j in {0,4,8} {
    \draw (12+\i,-14+\j)\DEC;
    }
    }
    \draw (18.5,-9.5) node {$\CA_T$};
    \draw (16.7,-7.5) node {$\CA_S$};
    \draw (16.5,-9.3) node {$d$};
    \draw (15.5,-9.3) node {1};
    \draw (16.5,-10.8) node {1};
    \end{scope}
    \end{scope}

    \begin{scope}[shift={(14,-54)}]
        \draw[ultra thick]plot [smooth,tension=1] coordinates {(3,-4.5) (-90:5.2) (-3,-4.5)};
	\begin{scope}[rotate around = {20:(-90:5.2)}]
	\draw[red] (-90:5.2) .. controls +(-60:4) and +(90:1) .. (3,-11.5)
	.. controls +(-90:2) and +(0:1) .. (-90:14)
	.. controls +(180:1) and +(-90:2) .. (-3,-11.5)
	.. controls +(90:1) and +(-120:4) .. (-90:5.2);
	\end{scope}
	\draw[rotate around = {35:(-90:5.2)}] (-90:7) node {$c$} ;
	\filldraw[white, fill = white] (0,-13.2) circle (10pt);
	\begin{scope}[rotate around = {-20:(-90:5.2)}]
	\draw[red] (-90:5.2) .. controls +(-60:4) and +(90:1) .. (3,-11.5)
	.. controls +(-90:2) and +(0:1) .. (-90:14)
	.. controls +(180:1) and +(-90:2) .. (-3,-11.5)
	.. controls +(90:1) and +(-120:4) .. (-90:5.2);
	\end{scope}
	\draw[rotate around = {-35:(-90:5.2)}] (-90:7) node {$d$} ;
	\draw (0,-8) node {1} ;
	\draw (-90:5.2)\DEC;
    \draw (-90:3) node{Figure 3.5};
    \draw (-4,-14) node {$\CA_T$} ;
    \draw (4,-14) node {$\CA_S$} ;
    \draw (8,-10) node {$\Longleftrightarrow$};

    \filldraw[fill=gray!20] (12,-6) -- (20,-6) -- (20,-14) -- (16,-14) -- (16,-10) -- (12,-10) -- (12,-6);
    \foreach \i in {0,4} {
    \foreach \j in {0,4} {
    \draw (12+\i,-10+\j) -- (12+\i,-14+\j);
    \draw (12+\i,-10+\j) -- (16+\i,-10+\j);
    \draw[cyan, ultra thick] (12,-10) .. controls +(-80:3) and +(170:3) .. (16,-14);
    \begin{scope}[rotate around ={45:(16+\i,-10+\j)}]
    \filldraw[thick, fill=gray] (16+\i,-10+\j) .. controls +(145:0.5) and +(90:0.3) .. (14.5+\i,-10+\j)
	.. controls +(-90:0.3) and +(-145:0.5) .. (16+\i,-10+\j);
    \end{scope}
    }
    }
    \draw (20,-6) -- (20,-14);
    \draw (12,-14) -- (20,-14);
    \foreach \i in {0,4,8} {
    \foreach \j in {0,4,8} {
    \draw (12+\i,-14+\j)\DEC;
    }
    }
    \draw (18.5,-9.5) node {$\CA_S$};
    \draw (16.7,-7.5) node {$\CA_T$};
    \draw (16.5,-9.3) node {$1$};
    \draw (15.5,-9.3) node {$c$};
    \draw (16.5,-10.8) node {$d$};
    \end{scope}
    \end{tikzpicture}
	\caption{Binary-free cases}\label{fig:001}
    \end{figure}

\begin{figure}[htbp]
\begin{tikzpicture}[scale=.3]
\begin{scope}[shift={(6,0)}]
	\draw[ultra thick]plot [smooth,tension=1] coordinates {(3,-4.5) (-90:5.2) (-3,-4.5)};
	\draw[ultra thick,fill=gray!10] (-90:11) ellipse (1.5);
	\draw[red] (-90:5.2) .. controls +(-45:6) and +(-30:5) .. (-90:12.5);
	\draw[red] (-90:5.2) .. controls +(-135:6) and +(-150:5) .. (-90:12.5);
    \draw (3.7,-10) node {$\CA_S$} ;
    \draw (-4,-13) node {$\CA_S'$} ;
    \draw (4,-7) node {$\CA_T$} ;
	\draw[ultra thick]plot [smooth,tension=1] coordinates {(5.4,-12) (5,-10.5) (5.4,-9)};
	\draw[red] (-90:5.2) .. controls +(-20:3) and +(120:3) .. (5,-10.5);
	\draw[cyan, ultra thick] (0,-12.5) .. controls +(-60:3) and +(-120:3) .. (5,-10.5);
	\draw (-90:12.5)\DEC;
	\draw (-90:5.2)\DEC;
    \draw[blue] (-90:9.5)\nn;
    \draw (5,-10.5)\DEC;
    \draw (-90:11) node{$\Vot$};
	\draw (2,-7) node {1} ;
	\draw (-90:5.2)\DEC;
    \draw (-90:3) node{Figure 1.1.1};
\end{scope}

    \begin{scope}[shift={(26,0)}]
        \draw[ultra thick]plot [smooth,tension=1] coordinates {(3,-4.5) (-90:5.2) (-3,-4.5)};
	\draw[ultra thick,fill=gray!10] (-90:11) ellipse (1.5);
	\draw[red] (-90:5.2) .. controls +(-45:6) and +(-30:5) .. (-90:12.5);
	\draw[red] (-90:5.2) .. controls +(-135:6) and +(-150:5) .. (-90:12.5);
    \draw (3,-8) node {$\CA_S'$} ;
    \draw (-2,-14) node {$\CA_S$} ;
    \draw (4,-10.7) node {$\CA_T$} ;
	\draw[ultra thick]plot [smooth,tension=1] coordinates {(5.4,-12) (5,-10.5) (5.4,-9)};
	\draw[red] (0,-12.5) .. controls +(-60:4) and +(-120:3) .. (5,-10.5);
	\draw[cyan, ultra thick] (0,-5.2) .. controls +(-145:14) and +(-100:11) .. (5,-10.5);
	\draw (-90:12.5)\DEC;
	\draw (-90:5.2)\DEC;
    \draw[blue] (-90:9.5)\nn;
    \draw (5,-10.5)\DEC;
    \draw (-90:11) node{$\Vot$};
	\draw (0,-13.5) node {1} ;
	\draw (-90:5.2)\DEC;
    \draw (-90:3) node{Figure 1.1.2};
    \end{scope}

    \begin{scope}[shift={(0,-16)}]
        \draw[ultra thick]plot [smooth,tension=1] coordinates {(3,-4.5) (-90:5.2) (-3,-4.5)};
	\draw[ultra thick,fill=gray!10] (-90:11) ellipse (1.5);
	\draw[red] (-90:5.2) .. controls +(-45:6) and +(-30:5) .. (-90:12.5);
	\draw[red] (-90:5.2) .. controls +(-135:6) and +(-150:5) .. (-90:12.5);
    \draw (2.5,-14) node {$\CA_S$} ;
    \draw (-4,-13) node {$\CA_S'$} ;
    \draw (4,-7) node {$\CA_T$} ;
	\draw[red] (-90:12.5) ..
	controls +(-145:8) and +(-90:8) ..
	(4.5,-11) ..
	controls +(90:3) and +(-20:3) ..
	(0,-5.2);
	\draw[cyan, ultra thick] (0,-12.5) .. controls +(-120:3) and +(180:1) ..
	(0,-16) ..
	controls +(0:1) and +(-60:3) ..
	(0,-12.5);
	\draw (-90:12.5)\DEC;
	\draw (-90:5.2)\DEC;
    \draw[blue] (-90:9.5)\nn;
    \draw (-90:11) node{$\Vot$};
	\draw (2,-7) node {1} ;
	\draw (0,-14) node {$d$};
	\draw (-90:5.2)\DEC;
    \draw (-90:3) node{Figure 1.2.1.1};
    \end{scope}

\begin{scope}[shift={(16,-16)}]
        \draw[ultra thick]plot [smooth,tension=1] coordinates {(3,-4.5) (-90:5.2) (-3,-4.5)};
	\draw[ultra thick,fill=gray!10] (-90:11) ellipse (1.5);
	\draw[red] (-90:5.2) .. controls +(-45:6) and +(-30:5) .. (-90:12.5);
	\draw[red] (-90:5.2) .. controls +(-135:6) and +(-150:5) .. (-90:12.5);
    \draw (2.5,-14) node {$\CA_S'$} ;
    \draw (-4,-13) node {$\CA_S$} ;
    \draw (5.4,-10) node {$\CA_T$} ;
	\draw[red] (-90:12.5) ..
	controls +(-90:5) and +(-90:5) ..
	(4.5,-11) ..
	controls +(90:3) and +(-20:3) ..
	(0,-5.2);
	\draw[cyan, ultra thick] (0,-5.2)
        .. controls +(-170:9) and +(170:6) .. (-3,-16)
        .. controls +(-10:1) and +(-170:1) .. (3,-16)
        .. controls +(10:6) and +(-10:9) .. (0,-5.2);
	\draw (-90:12.5)\DEC;
	\draw (-90:5.2)\DEC;
    \draw[blue] (-90:9.5)\nn;
    \draw (-90:11) node{$\Vot$};
	\draw (-0.7,-13.6) node {1} ;
	\draw[Green,->, thick,>=stealth, shift={(0,-5.2)}]plot [smooth,tension=1] coordinates
    {(-20:1) (-72:1.1) (-135:1)};
	\draw (0.4,-6.8) node {$d$};
	\draw (-90:5.2)\DEC;
    \draw (-90:3) node{Figure 1.2.1.2};
    \end{scope}

    \begin{scope}[shift={(32,-16)}]
	\draw[ultra thick]plot [smooth,tension=1] coordinates {(3,-4.5) (-90:5.2) (-3,-4.5)};
	\draw[red] (0,-5.2)
        .. controls +(-170:7) and +(180:6) .. (0,-14)
        .. controls +(0:4) and +(-50:3) .. (4.7,-10.6);
	\begin{scope}[rotate around = {40:(-90:5.2)}]
	\draw[ultra thick,fill=gray!10] (-90:11) ellipse (1.5);
	\draw[red] (-90:5.2) .. controls +(-45:6) and +(-30:5) .. (-90:12.5);
	\draw[red] (-90:5.2) .. controls +(-135:6) and +(-150:5) .. (-90:12.5);
	\draw (-90:12.5)\DEC;
	\draw[blue] (-90:9.5)\nn;
	\end{scope}
	\begin{scope}[rotate around = {-40:(-90:5.2)}]
	\draw[cyan, ultra thick] (-90:5.2) .. controls +(-70:2) and +(90:1) .. (1.4,-9)
	.. controls +(-90:1) and +(0:1) .. (-90:11)
	.. controls +(180:1) and +(-90:1) .. (-1.4,-9)
	.. controls +(90:1) and +(-110:2) .. (-90:5.2);
	\end{scope}
	\draw[rotate around = {-40:(-90:5.2)}] (-90:7) node {$d$} ;
	\draw (4.6,-12) node {1} ;
	\draw (-90:5.2)\DEC;
    \draw (-90:3) node{Figure 1.2.1.3};
    \draw (0.5,-12) node {$\CA_S$} ;
    \draw (7,-7.2) node {$\CA_S'$} ;
    \draw (2,-14.7) node {$\CA_T$} ;
    \draw (3.8,-9.7) node {$\Vot$} ;
    \end{scope}

    \begin{scope}[shift={(0,-32)}]
        \draw[ultra thick]plot [smooth,tension=1] coordinates {(3,-4.5) (-90:5.2) (-3,-4.5)};
	\draw[ultra thick,fill=gray!10] (-90:11) ellipse (1.5);
	\draw[red] (-90:5.2) .. controls +(-45:6) and +(-30:5) .. (-90:12.5);
	\draw[red] (-90:5.2) .. controls +(-135:6) and +(-150:5) .. (-90:12.5);
	\draw[red] (-90:5.2) .. controls +(-20:4) and +(80:2) .. (5,-11)
	.. controls +(-100:3) and +(0:2) .. (0,-14.5)
	.. controls +(180:2) and +(-80:3) .. (-5,-11)
	.. controls +(100:2) and +(-160:4) .. (-90:5.2);
	\draw[cyan, ultra thick] (0,-5.2) .. controls +(-145:3) and +(90:3) ..
	(-4,-11) ..
	controls +(-90:2) and +(-120:3) ..
	(0,-12.5);
	\draw[Green,->, thick,>=stealth, shift={(0,-5.2)}]plot [smooth,tension=1] coordinates
    {(-45:1.1) (-102:1.2) (-160:1.1)};
    \draw (-90:12.5)\DEC;
	\draw (-90:5.2)\DEC;
    \draw[blue] (-90:9.5)\nn;
    \draw (4,-11) node {$\CA_S$} ;
    \draw (-2.5,-11) node {$\CA_S'$} ;
    \draw (4.5,-7) node {$\CA_T$} ;
    \draw (-90:11) node{$\Vot$};
	\draw (2,-6.7) node {1} ;
	\draw (-0.5,-7) node {$d$} ;
	\draw (-90:5.2)\DEC;
    \draw (0,-3) node{Figure 1.2.2.1};
    \end{scope}

    \begin{scope}[shift={(16,-32)}]
	\draw[ultra thick]plot [smooth,tension=1] coordinates {(3,-4.5) (-90:5.2) (-3,-4.5)};
	\draw[cyan, ultra thick] (-4.7,-10.7)
        .. controls +(-100:4) and +(180:4) .. (3,-14)
        .. controls +(0:6) and +(-10:9) .. (0,-5.2);
	\begin{scope}[rotate around = {-40:(-90:5.2)}]
	\draw[ultra thick,fill=gray!10] (-90:11) ellipse (1.5);
	\draw[red] (-90:5.2) .. controls +(-45:6) and +(-30:5) .. (-90:12.5);
	\draw[red] (-90:5.2) .. controls +(-135:6) and +(-150:5) .. (-90:12.5);
	\draw (-90:12.5)\DEC;
	\draw[blue] (-90:9.5)\nn;
	\end{scope}
	\begin{scope}[rotate around = {40:(-90:5.2)}]
	\draw[red] (-90:5.2) .. controls +(-60:4) and +(90:1) .. (2.5,-11.5)
	.. controls +(-90:2) and +(0:1) .. (-90:14)
	.. controls +(180:1) and +(-90:2) .. (-2.5,-11.5)
	.. controls +(90:1) and +(-120:4) .. (-90:5.2);
	\end{scope}
	\draw[rotate around = {40:(-90:5.2)}] (-90:7) node {$d$} ;
	\draw (0.2,-8) node {1} ;
	\draw (-90:5.2)\DEC;
    \draw (-90:3) node{Figure 1.2.2.2};
    \draw (-0.3,-12) node {$\CA_S$} ;
    \draw (-7.5,-10) node {$\CA_S'$} ;
    \draw (2,-12.5) node {$\CA_T$} ;
    \draw (-3.8,-9.5) node {$\Vot$} ;
    \end{scope}

\begin{scope}[shift={(32,-32)}]
        \draw[ultra thick]plot [smooth,tension=1] coordinates {(3,-4.5) (-90:5.2) (-3,-4.5)};
	\draw[ultra thick,fill=gray!10] (-90:11) ellipse (1.5);
	\draw[red] (-90:5.2) .. controls +(-45:6) and +(-30:5) .. (-90:12.5);
	\draw[red] (-90:5.2) .. controls +(-135:6) and +(-150:5) .. (-90:12.5);
    \draw (-1.7,-14.8) node {$\CA_T$} ;
    \draw (-3.6,-10) node {$\CA_S$} ;
    \draw (4,-10) node {$\CA_S'$} ;
	\draw[cyan, ultra thick] (-90:12.5) ..
	controls +(-35:8) and +(-90:8) ..
	(-4.5,-11) ..
	controls +(90:3) and +(-160:3) ..
	(0,-5.2);
	\draw[red] (0,-12.5) .. controls +(-120:3) and +(180:1) ..
	(0,-16) ..
	controls +(0:1) and +(-60:3) ..
	(0,-12.5);
	\draw (-90:12.5)\DEC;
	\draw (-90:5.2)\DEC;
    \draw[blue] (-90:9.5)\nn;
    \draw (-90:11) node{$\Vot$};
	\draw (-1.2,-13.7) node {1} ;
	\draw (0,-14) node {$d$};
	\draw (-90:5.2)\DEC;
    \draw (-90:3) node{Figure 1.2.2.3};
    \end{scope}
    
    \begin{scope}[shift={(0,-48)}]
	\draw[ultra thick]plot [smooth,tension=1] coordinates {(3,-4.5) (-90:5.2) (-3,-4.5)};
	\draw[cyan, ultra thick] (-5.2,-10.5)
        .. controls +(-100:4) and +(180:4) .. (0,-14)
        .. controls +(0:4) and +(-80:4) .. (5.2,-10.5);
	\begin{scope}[rotate around = {-45:(-90:5.2)}]
	\draw[ultra thick,fill=gray!10] (-90:11) ellipse (1.5);
	\draw[red] (-90:5.2) .. controls +(-50:6) and +(-30:5) .. (-90:12.5);
	\draw[red] (-90:5.2) .. controls +(-130:6) and +(-150:5) .. (-90:12.5);
	\draw (-90:12.5)\DEC;
	\draw[blue] (-90:9.5)\nn;
	\draw (3.8,-10) node {$\CA_S$} ;
    \draw (-3.8,-10) node {$\CA_S'$} ;
	\end{scope}
	\begin{scope}[rotate around = {45:(-90:5.2)}]
	\draw[ultra thick,fill=gray!10] (-90:11) ellipse (1.5);
	\draw[red] (-90:5.2) .. controls +(-50:6) and +(-30:5) .. (-90:12.5);
	\draw[red] (-90:5.2) .. controls +(-130:6) and +(-150:5) .. (-90:12.5);
	\draw (-90:12.5)\DEC;
	\draw[blue] (-90:9.5)\nn;
	\draw (-3.8,-11) node {$\CA_T$} ;
    \draw (3.8,-11) node {$\CA_T'$} ;
	\end{scope}
	\draw (0,-8) node {1} ;
	\draw (-90:5.2)\DEC;
    \draw (-90:3) node{Figure 2.1};
    \draw (-4,-9.4) node {$\Vot_1$} ;
    \draw (4.2,-9.4) node {$\Vot_2$} ;
    \end{scope}

     \begin{scope}[shift={(16,-48)}]
	\draw[ultra thick]plot [smooth,tension=1] coordinates {(3,-4.5) (-90:5.2) (-3,-4.5)};
    \draw[cyan, ultra thick] (0,-5.2)
        .. controls +(-30:6) and +(90:3)
        .. (6,-16)
        .. controls +(-90:4) and +(-60:5) .. (0,-19.8);
	\draw[ultra thick,fill=gray!10] (-90:11) ellipse (1.5);
	\draw[red] (-90:5.2) .. controls +(-50:6) and +(-30:5) .. (-90:12.5);
	\draw[red] (-90:5.2) .. controls +(-130:6) and +(-150:5) .. (-90:12.5);
	\draw[blue] (-90:9.5)\nn;
	\draw (3.6,-10) node {$\CA_T$} ;
    \draw (-3.8,-10) node {$\CA_T'$} ;
	\begin{scope}[shift = {(0,-7.3)}]
	\draw[ultra thick,fill=gray!10] (-90:11) ellipse (1.5);
	\draw[red] (-90:5.2) .. controls +(-50:6) and +(-30:5) .. (-90:12.5);
	\draw[red] (-90:5.2) .. controls +(-130:6) and +(-150:5) .. (-90:12.5);
	\draw[blue] (-90:9.5)\nn;
	\draw (-3.8,-10) node {$\CA_S'$} ;
    \draw (3.6,-10) node {$\CA_S$} ;
	\end{scope}
	\draw (1.6,-13.7) node {1} ;
	\draw (-90:5.2)\DEC;
	\draw (-90:12.5)\DEC;
	\draw (-90:19.8)\DEC;
    \draw (-90:3) node{Figure 2.2};
    \draw (0,-11) node {$\Vot_1$} ;
    \draw (0,-18.3) node {$\Vot_2$} ;
    \end{scope}

    \begin{scope}[shift={(32,-48)}]
	\draw[ultra thick]plot [smooth,tension=1] coordinates {(3,-4.5) (-90:5.2) (-3,-4.5)};
    \draw[cyan, ultra thick] (0,-5.2)
        .. controls +(-150:6) and +(90:3)
        .. (-6,-16)
        .. controls +(-90:4) and +(-120:5) .. (0,-19.8);
	\draw[ultra thick,fill=gray!10] (-90:11) ellipse (1.5);
	\draw[red] (-90:5.2) .. controls +(-50:6) and +(-30:5) .. (-90:12.5);
	\draw[red] (-90:5.2) .. controls +(-130:6) and +(-150:5) .. (-90:12.5);
	\draw[blue] (-90:9.5)\nn;
	\draw (3.8,-10) node {$\CA_S'$} ;
    \draw (-3.5,-10) node {$\CA_S$} ;
	\begin{scope}[shift = {(0,-7.3)}]
	\draw[ultra thick,fill=gray!10] (-90:11) ellipse (1.5);
	\draw[red] (-90:5.2) .. controls +(-50:6) and +(-30:5) .. (-90:12.5);
	\draw[red] (-90:5.2) .. controls +(-130:6) and +(-150:5) .. (-90:12.5);
	\draw[blue] (-90:9.5)\nn;
	\draw (-3.5,-10) node {$\CA_T$} ;
    \draw (3.8,-10) node {$\CA_T'$} ;
	\end{scope}
	\draw (-1.6,-13.7) node {1} ;
	\draw (-90:5.2)\DEC;
	\draw (-90:12.5)\DEC;
	\draw (-90:19.8)\DEC;
    \draw (-90:3) node{Figure 2.3};
    \draw (0,-11) node {$\Vot_1$} ;
    \draw (0,-18.3) node {$\Vot_2$} ;
    \end{scope}
    \end{tikzpicture}
    \caption{Binary cases}
\label{fig:2}
    \end{figure}

\clearpage


\end{document}